\documentclass[12pt,dvips]{amsart}
\usepackage{euler, amsfonts, amssymb, latexsym, epsfig,epic} %,mathabx}

\usepackage[backref]{hyperref}
   \renewcommand*{\backrefalt}[4]{%
     \ifcase #1 %
     \or
     \vspace{0pt} \footnotesize (In \S#2)%
     \else
     \vspace{0pt} \footnotesize (In \S#2)
     \fi}
\renewcommand*{\backref}[1]{}

\newcommand\junk[1]{}
\junk{
  \usepackage[OT2,OT1]{fontenc}
  \newcommand\cyr{%
    \renewcommand\rmdefault{wncyr}%
    \renewcommand\sfdefault{wncyss}%
    \renewcommand\encodingdefault{OT2}%
    \normalfont
    \selectfont}
  \DeclareTextFontCommand{\textcyr}{\cyr}
}

\newcommand\notdivides{{\!\not\divides\,}}
\newcommand\divides{{\mid}}
\newcommand\lie[1]{{\mathfrak #1}}

\newcommand\Tr{{\rm Tr\,}}

\newcommand\Proj{{\rm Proj\,}}
\newcommand\iso{{\,\cong\,}}
\newcommand\tensor{{\otimes}}

\newcommand\calO{{\mathcal O}}
\newtheorem{Theorem}{Theorem} 
\newtheorem{Proposition}{Proposition} 
\newtheorem{Lemma}{Lemma}

\newtheorem{Corollary}{Corollary}
\newtheorem*{Corollary*}{Corollary}
\newtheorem*{Lemma*}{Lemma}
 
\newtheorem*{Theorem*}{Theorem}
\theoremstyle{remark}
\newtheorem{Example}{Example}

\newcommand\onto{\mathop{\twoheadrightarrow}}

\newcommand\into{\operatorname*{\hookrightarrow}}

\newcommand\union{\cup}
\newcommand\Union{\bigcup}

\newcommand\Pone{{\mathbb P}^1}

\newcommand\PP{{\mathbb P}}

\newcommand\integers{{\mathbb Z}}
\newcommand\rationals{{\mathbb Q}}

\newcommand\naturals{{\mathbb N}}

\newcommand\Fun{{\rm Fun}}
\newcommand\nulset{\emptyset}

\theoremstyle{plain}

\theoremstyle{remark}

\renewenvironment{quotation}
{\list{}{%\setlength\listparindent{0.5em}%
    \setlength\itemindent{0em}%
    \setlength\leftmargin{1.5em}
    \setlength\rightmargin{1.5em}
  }%
\item[]}
{\endlist}

\newcommand\defn[1]{{\bf #1}} % maybe should be \em

\newcommand\Spec{{\rm Spec}\,}
\newcommand\Gm{{{\mathbb G}_m}}
\newcommand\Ga{{{\mathbb G}_a}}

\newcommand\QQ{\rationals}

\newcommand\FF{{\mathbb F}}
\newcommand\barX{{\overline{X}}}

\newcommand\lt{{\rm init\space}}

\begin{document}
\pagestyle{plain}

\title{Frobenius splitting, point-counting, and degeneration}

\author{Allen Knutson}
\thanks{AK was supported by NSF grant 0303523.}
\email{allenk@math.cornell.edu}
\date{\today}

\maketitle

\renewcommand\AA{{\mathbb A}}

\begin{abstract}
  Let $f$ be a polynomial of degree $n$ in $\integers[x_1,\ldots,x_n]$, 
  typically reducible but squarefree. From the hypersurface $\{f=0\}$ 
  one may construct a number of other subschemes $\{Y\}$ by extracting
  prime components, taking intersections, taking unions, and iterating
  this procedure.

  We prove that if the number of solutions to $f=0$ in $\FF_p^n$
  is not a multiple of $p$, then
  all these intersections in $\AA^n_{\FF_p}$ just described are reduced.
  (If this holds for infinitely many $p$, then it holds over $\QQ$ as well.) 
  More specifically, there is a {\em Frobenius splitting} on
  $\AA^n_{\FF_p}$ compatibly splitting all these subschemes $\{Y\}$.

  We determine when a Gr\"obner degeneration $f_0=0$ of such a hypersurface
  $f=0$ is again such a hypersurface. Under this condition, we prove
  that compatibly split subschemes degenerate to compatibly split subschemes,
  and in particular, stay reduced.

  Together these suggest that the number of $\FF_p$-points
  on the general fiber $Y$ and special fiber $Y'$ of a Gr\"obner degeneration
  should, in good cases, differ by a multiple of $p$.
  Under very special Gr\"obner degenerations 
  (``geometric vertex decompositions''),
  we give a discontinuous injection of $Y$ into $Y'$ that lets us
  compare the relate their classes in the Grothendieck group of varieties,
  and thereby demonstrate this.

  Our results are strongest in the case that $f$'s lexicographically first term 
%  that one can weight the variables $(x_i)$ such 
  is $\prod_{i=1}^n x_i$. Then for all large $p$, % \notdivides c$,
  there is a Frobenius splitting that compatibly splits $f$'s hypersurface
  and all the associated $\{Y\}$.
  The Gr\"obner degeneration $Y'$ of each such $Y$ 
  is a reduced union of coordinate spaces (a Stanley-Reisner scheme),
  and we give a result to help compute its Gr\"obner basis.
  We exhibit an $f$ whose associated $\{Y\}$ include Fulton's matrix 
  Schubert varieties, and recover much more easily 
  the Gr\"obner basis theorem of [Knutson-Miller '05].
  We show that in Bott-Samelson
  coordinates on an opposite Bruhat cell $X^v_\circ$ in $G/B$, the $f$ defining
  the complement of the big cell also has initial term $\prod_{i=1}^n x_i$,
  and hence the Kazhdan-Lusztig subvarieties $\{X^v_{w\circ}\}$
  degenerate to Stanley-Reisner schemes. This recovers, in a weak
  form, the main result of [Knutson '08].
\end{abstract}

{\Small
  \setcounter{tocdepth}{1}
  \tableofcontents
}

\section{Introduction, and statement of results}

A commutative ring $R$ is \defn{reduced} if it has no nilpotents, i.e.,
if for $m>1$ the map $r \mapsto r^m$ takes only $0$ to $0$. It is tempting
to write this as $\ker (r\mapsto r^m) = 0$, and one may indeed do so
if $m$ is a prime $p$ and $R$ contains the field $\FF_p$ of $p$ elements.
Then the \defn{Frobenius map} $r\mapsto r^p$ is $\FF_p$-linear, and
the condition of $R$ being reduced says that this map has a one-sided inverse.
This, for us, motivates the study of these inverses.

Define a \defn{(Frobenius) splitting} \cite{BK} of a commutative
$\FF_p$-algebra $R$ as a map $\varphi: R\to R$ satisfying three conditions:
\begin{itemize}
\item $\varphi(a+b) = \varphi(a) + \varphi(b)$ 
\item $\varphi(a^p b) = a\, \varphi(b)$
\item $\varphi(1)=1$.
\end{itemize}
If $\varphi$ only satisfies the first two conditions (e.g. $\varphi \equiv 0$), 
we will call it a \defn{near-splitting}\footnote{%
  One might object that $0$ is not so near to being a splitting.
  Such maps have also been called ``$p^{-1}$-linear'' and probably
  many other things.}. In section \ref{sec:nearsplittings}
we will recall from \cite[section 1.3.1]{BK} the classification
of near-splittings of affine space. 

If $R$ is equipped with a splitting $\varphi$, we will say $R$ is \defn{split}
(not just ``splittable''; we care about the choice of $\varphi$).
Call an ideal $I\leq R$ of a ring with a Frobenius (near-)splitting $\varphi$ 
\defn{compatibly (near-)split} if $\varphi(I) \subseteq I$. 
For the convenience of the reader we recapitulate the basic results
of Frobenius splitting we will use:

\begin{Theorem*}\cite[section 1.2]{BK}
  Let $R$ be a Frobenius split ring with ideals $I,J$.
  \begin{enumerate}
  \item $R$ is reduced.
  \item If $I$ is compatibly split, then $I$ is radical, 
    and $\varphi(I)=I$.
  \item If $I$ and $J$ are compatibly split ideals, then so are
    $I\cap J$ and $I + J$. Hence they are radical.
  \item If $I$ is compatibly split, and $J$ is arbitrary, 
    then $I:J$ is compatibly split. In particular the prime
    components of $I$ are compatibly split.
  \end{enumerate}
\end{Theorem*}

Note that the sum of radical ideals is frequently {\em not} radical;
``compatibly split'' is a much more robust notion.

\begin{proof}
  \begin{enumerate}
  \item Assume not, and let $r$ be a nonzero nilpotent with $m$
    chosen largest such that $r^m \neq 0$ but $r^{m+1}=0$. 
    Let $s=r^m$. Then $0=s^p$, so $0=\varphi(s^p)=s$, contradiction.
  \item If $I$ is compatibly split, then $\varphi$ descends to a
    splitting of $R/I$, so $R/I$ is reduced. Equivalently, $I$ is radical.
    Since $I$ contains $\{i^p : i\in I\}$, 
    one always has $\varphi(I) \supseteq I$.
  \item 
    $\varphi(I\cap J) \subseteq \varphi(I)\cap \varphi(J) \subseteq I\cap J$. 
    $\varphi(I+J)\subseteq \varphi(I)+\varphi(J)$ because $\varphi$ is additive.
  \item 
%    \begin{align*}
$      r \in I:J \iff \forall j\in J, rj \in I 
      \implies \forall j\in J, rj^p \in I 
      \implies \forall j\in J, \varphi(rj^p) \in I$
      (since $I$ is compatibly split)
      $\iff \forall j\in J, \varphi(r)j \in I 
      \iff  \varphi(r) \in I:J. $
%    \end{align*}
  \end{enumerate}
\end{proof}

If $\varphi$ is a near-splitting such that $\varphi(1)$ is not a zero divisor, 
then parts 1,3,4 of this theorem still hold.
Unlike splitting and near-splitting, this notion does not always
pass to $R/I$; the induced near-splitting $\varphi'$ 
on $R/I$ may have $\varphi'(1)$ being a zero divisor,
in which case part 2 of the theorem can fail.

\begin{Corollary}\label{cor:allradical}
  Let $I$ be a compatibly split ideal in a Frobenius split ring.
  From it we can construct many more ideals, by taking prime components,
  sums, and intersections, then iterating. All of these will be radical.
\end{Corollary}

It was recently observed \cite{Schwede,KumarMehta}, and 
only a little harder to prove (a few pages, rather than a few lines),
that a Noetherian split ring $R$ has only finitely many compatibly
split ideals. In very special cases the algorithm suggested in
corollary \ref{cor:allradical} finds all of them.

%Let $\FF$ be a perfect field over $\FF_p$.
As we recall in section \ref{sec:nearsplittings},
there is a near-splitting on $\FF_p[x_1,\ldots,x_n]$ called $\Tr(\bullet)$ 
uniquely characterized by its application to monomials $m$:
$$ \Tr(m) =   \begin{cases}
  \sqrt[p]{m \prod_i x_i} \bigg/ \prod_i x_i 
  & \text{if $m \prod_i x_i$ is a $p$th power} \\
  0 & \text{otherwise.}
\end{cases} $$
The \defn{standard splitting} of $\FF_p[x_1,\ldots,x_n]$ is 
$\varphi(g) := \Tr\left( (\prod_{i=1}^n x_i)^{p-1} g\right)$.

\begin{Lemma}\label{lem:coordspaces}
%  Let $\FF$ be a perfect field and $\varphi$ the 
  The standard splitting %  on $\FF[x_1,\ldots,x_n]$.
  is a Frobenius splitting, and the ideals that it
  compatibly splits are exactly the \defn{Stanley-Reisner ideals} 
  (meaning, those generated by squarefree monomials).
\end{Lemma}

Occasionally we will need the near-splittings $\Tr(\bullet)$ defined
on the coordinate rings of different affine spaces at the same time;
in this case we will use subscripts to avoid confusion, e.g. 
$\Tr_H$ vs. $\Tr_{H\times L}$. 

\subsection{Point-counting over $\FF_p$ and Frobenius splitting}
Our first result relates these. % two.

\begin{Theorem}\label{thm:Trpointcount}
  Let $f \in \FF_p[x_1,\ldots,x_n]$ be a polynomial of degree at most $n>0$.
  Then the number of points $\vec v \in \FF_p^n$ in the affine hypersurface
  defined by $f=0$ is congruent to $(-1)^{n-1}\, \Tr(f^{p-1})$. 

  If this number is not a multiple of $p$, then some multiple of
  $\Tr(f^{p-1}\bullet)$ defines a Frobenius splitting on 
  $\FF_p[x_1,\ldots,x_n]$, with respect to which $\langle f \rangle$
  is compatibly split. (In this case $\deg f$ is indeed $n$, not less than $n$,
  by the Chevalley-Warning theorem.)

  If $f = \prod_{i=1}^m f_i$ where each $\deg f_i>0$, then the
  number of points $\vec v \in \FF_p^n$ in the subvariety defined
  by $f_1 = f_2 = \ldots = f_m = 0$
  is congruent to $(-1)^{n-m}\, \Tr(f^{p-1})$. 

  In particular, if the number of points in the $f=0$ hypersurface
  (or the $f_1 = f_2 = \ldots = f_m = 0$ subscheme, if $f$ factors)
  is not a multiple of $p$, then we can run the algorithm 
  in corollary \ref{cor:allradical} starting with the ideal $\langle f\rangle$, 
  and produce only radical ideals.
\end{Theorem}

\newcommand\init{\mathop{\rm init}\space}

If $\deg f < n$, then one can use the Chevalley-Warning theorem
to show that no multiple of $\Tr(f^{p-1}\bullet)$ defines a 
Frobenius splitting on $\FF_p[x_1,\ldots,x_n]$.
On the other hand, if the hypersurface defined by $f=0$ is smooth --
regardless of $\deg f$ --
then there is {\em some} splitting of affine space that compatibly
splits $\langle f\rangle$ \cite[proposition 1.1.6]{BK}.

While we think that theorem \ref{thm:Trpointcount} provides an
interesting link between point-counting and reducedness, 
we don't have any real examples where the point-counting is
the easiest way to demonstrate the splitting. 
Theorem \ref{thm:frobdegen} part (2) and especially theorem \ref{thm:initprod}
provide more checkable sufficient conditions.

\subsection{Frobenius splitting and degeneration} 

Given a weighting $\lambda:\{1,\ldots,n\} \to \naturals$ on our
variables $(x_i)$, we can define the \defn{leading form} $\init(f)$ of
any polynomial $f$ as the sum of the terms $c \prod_i x_i^{e_i}$ with
maximum $\sum_i \lambda_i e_i$. It has a nice interpretation in
terms of the \defn{Newton polytope} of $f$, which is defined as
the convex hull of the exponent vectors of the monomials in $f$;
the weighting $\lambda$ defines a linear functional on the space of
exponent vectors, and it is maximized on one face $F$ of $f$'s
Newton polytope. ``Take the exponent vector'' is a map from the set of
$f$'s terms to the Newton polytope, and $\init(f)$ is the sum of the
terms lying over $F$.

One can also define $\init(I)$ for any ideal $I \leq R[x_1,\ldots,x_n]$
(where $R$ is any base ring) as the $R$-span of $\{\init(f) : f\in I\}$.
We mention that if $\{f_1,\ldots,f_m\} \subseteq I$ have the property that
$\{\init(f_i)\}$ generate $\init(I)$, then $\{f_1,\ldots,f_m\}$ is
called a \defn{Gr\"obner basis} for the pair $(I,\lambda)$.
We will not use much of the theory of Gr\"obner bases,
but direct the interested reader to \cite{Sturmfels96}.

\begin{Lemma}\label{lem:Trinit}
  For any polynomial $g$ and weighting $\lambda$, 
  $\Tr(\init(g))$ is either $0$ or $\init(\Tr(g))$. 
\end{Lemma}

\begin{Lemma}\label{lem:homog}
  If $f = \init f$, then for any subvariety $Y$ compatibly split by
  $\Tr(f^{p-1}\bullet)$, we have $Y = \init Y$. In particular, if $f$
  is homogeneous, then $Y$ is the affine cone over a projective variety
  and has a well-defined degree.
\end{Lemma}

\newcommand\YY{{\mathcal Y}}

\begin{Theorem}\label{thm:frobdegen}
  Let $f \in \FF_p[x_1,\ldots,x_n]$ be of degree at most $n$.
  
  If $\prod_i x_i$ is not in $f$'s Newton polytope, e.g. if
  $\deg f < n$, then $\Tr(f^{p-1})=0$. Hence no multiple of
  $\Tr(f^{p-1}\bullet)$ is a Frobenius splitting.
  Hereafter let $\lambda$ be a weighting such that $\prod_i x_i$ 
  (or some $\FF_p$-multiple) lies in $\init(f)$. 
  In particular $(1,1,\ldots,1)$ lies in $f$'s Newton polytope. 

  \begin{enumerate}
  \item 
    $\Tr(f^{p-1}) = \Tr(\init(f)^{p-1})$,
    so (a multiple of) $\Tr(f^{p-1}\bullet)$
    defines a Frobenius splitting iff (the same multiple of)
    $\Tr(\init(f)^{p-1}\bullet)$ does.
  \item Assume hereafter that some multiple of 
    $\Tr(f^{p-1}\bullet)$ and $\Tr(\init(f)^{p-1}\bullet)$ do define splittings.
    Let $I$ be an ideal compatibly split with respect to the first splitting.
    Then $\init(I)$ is compatibly split with respect to the second splitting.
  \item Let $\YY_f$ and $\YY_{\init f}$ denote the
    poset of irreducible varieties compatibly split by 
    $\Tr(f^{p-1}\bullet)$ and $\Tr(\init(f)^{p-1}\bullet)$ respectively,
    partially ordered by inclusion. Then the map     
    $$ \pi_{f,\init}: \YY_{\init f} \to \YY_f, \qquad 
    Y' \mapsto \text{the unique minimal $Y$ such that} \init Y \supseteq Y' $$
    is well-defined, order-preserving, and surjective.
    Moreover, if $Y_1' \in \YY_{\init f}$, $Y_2 \in \YY_f$, 
    and $Y_1 = \pi_{f,\init}(Y_1')$, then
    $$ Y_2 \supseteq Y_1 \qquad \iff \qquad 
    \exists\ Y_2'\supseteq Y_1',\ \pi_{f,\init}(Y_2')=Y_2
    \quad \text{where } Y_2'\in\YY_{\init f},     $$
    so the partial order on $\YY_f$ is determined by that on $\YY_{\init f}$
    plus the map $\pi_{f,\init}$. 
    \junk{
      If $f$ is homogeneous, then the quotient $\YY_f$ of $\YY_{\init f}$
      is determined up to unique isomorphism by the function 
      $\YY_{\init f}\to \naturals$, $Y' \mapsto \dim \pi_{f,\init}(Y')$.
    }
  \item Assume $f$ is homogeneous, and let
    $ \YY^{=\dim}_{\init f} 
    := \{ Y' \in \YY_{\init f} : \dim \pi_{f,\init}(Y') = \dim Y' \}$. \\
    Then for any $Y \in \YY_{\init f}$,
    $$ \sum_{Y' \in\ \pi_{f,\init}^{-1}(Y) \ \cap\ \YY^{=\dim}_{\init f}} 
    \deg Y' = \deg Y. $$
  \end{enumerate}
\end{Theorem}

Some examples of the poset maps are given in figure \ref{fig:elliptic}.
Note that conclusion (2) runs the opposite direction
of a standard principle, which is that
for any ideal $I$, if $\init(I)$ is radical, then $I$ is radical.

For any polynomial $f$ whose Newton polytope contains $\prod_i x_i$, 
there is a unique minimal face of the polytope that contains it, 
and a corresponding minimal $\init(f)$ (minimal in number of terms).  
In this sense it is enough to study hypersurfaces $f=0$ where $\prod_i x_i$ 
lies in the interior of $f$'s Newton polytope.

One can also allow $\lambda$ to take values in $\naturals[\varepsilon]$,
where $\varepsilon$ is interpreted as infinitesimally positive
(i.e. $1 > N_1\varepsilon > N_2\varepsilon^2 > \ldots > 0$
for any $N_1,N_2,\ldots \in \naturals_+$)
with which to break ties. This doesn't change any of the results;
indeed, for any fixed $I$ and any such $\lambda$, there is a
$\lambda'$ taking only $\naturals$-values with 
$\init_\lambda(I) = \init_{\lambda'}(I)$ \cite{Sturmfels96}.
One sort of $\lambda$ that will often interest us is
$\lambda=(0,\ldots,0,1,0,\ldots,0)$, which we may indicate by
writing $\init_i$ where the $1$ is in the $i$th place.

In theorems \ref{thm:Trpointcount} and \ref{thm:frobdegen} the
interesting case is when $\deg f = n$, and there is little change
if $f$ is replaced by its degree $n$ homogeneous component. (Indeed, this is
the $\lambda = (1,1,\ldots,1)$ case of theorem \ref{thm:frobdegen}.)
Then $f=0$ defines an anticanonical hypersurface of $\PP^{n-1}$,
so, when smooth, a Calabi-Yau hypersurface. In the $n=3$ case,
this is an elliptic curve, split for infinitely many $p$
(see e.g. \cite{DP}).
However, the hypersurfaces that interest us are typically highly
reducible and in particular, singular.

Theorems \ref{thm:Trpointcount} and \ref{thm:frobdegen} taken
together show that certain Gr\"obner degenerations (meaning, 
replacements of $f$ by $\init(f)$) of a hypersurface
don't change the number of $\FF_p$-solutions, mod $p$. However, the
number of solutions does indeed change. For example, $xy=1$ 
has $p-1$ solutions in $\FF_p^2$, whereas $xy=0$ has $2p-1$.

\begin{figure}[h]
  \centering
  \epsfig{file=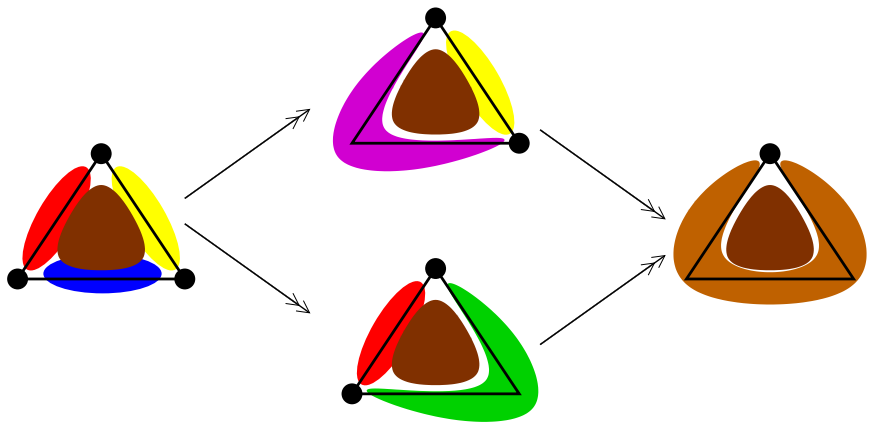,width=4.5in}
  \caption{The posets $\YY_f$ (minus each one's minimal element, $\{\vec 0\}$)
    defined in theorem \ref{thm:frobdegen} part (3)
    for 
    $f = xyz$ (left), $f = y(xz + y^2)$ (top), $f = z(xy + z^2)$ (bottom), 
    and $f = xyz + y^3 + z^3$ (right) drawn as identifications of the %boolean 
    lattice of faces of a $2$-simplex.
    The maps between them come from evident choices of $\init$.
  }
  \label{fig:elliptic}
\end{figure}

\subsection{Degeneration and point-counting over $\FF_p$} 

We study a very special kind of degeneration in this section,
that we called a {\em geometric vertex decomposition} 
in \cite{KMY,automatically}. 

Let $X \subseteq \AA^n$ be reduced and irreducible,
and split $\AA^n$ as a Cartesian product $H\times L$, standing for
Hyperplane and Line. Let $\Gm$ act on $\AA^n$ by scaling the coordinate on $L$,
i.e.
$$ z\cdot (h,\ell) = (h,z\ell),
\qquad \text{and define} \qquad X' := \lim_{t\to 0} t\cdot X $$
using this action. 

It is quite easy to determine the limit scheme \footnote{%
  This is defined in the usual way, as the zero fiber of the closure of 
  $\Union_{z\in\Gm} \left( \{z\} \times (z\cdot X) \right) 
  \subseteq \AA^1 \times (H\times L)$.
  One can also consider the limit {\em branchvariety} 
  as in \cite{automatically}, but this will coincide with the limit scheme
  under conditions (1) or (2) in the theorem.}
$X'$ as a set.
%Let $\Pi^\circ \subseteq H$ be the image of the projection of $X$ to $H$,
%and $\Pi$ be its closure.
Let $\Pi \subseteq H$ be the closure of 
the image of the projection of $X$ to $H$,
let $\barX$ be the closure of $X$ inside $H \times (L \union \{\infty\})$, and 
define $\Lambda\subseteq H$ by
$\Lambda \times \{\infty\} := \barX \cap (H \times \{\infty\}) 
= \barX \setminus X$. Then
$$ X' = (\Pi \times \{0\}) \cup (\Lambda \times L)$$
as a set \cite[theorem 2.2]{KMY}.

Though it was not pointed out in \cite{KMY}, %,automatically},
none of this changes if $H$ is allowed to be an arbitrary scheme $H'$
(though $L$ must remain $\AA^1$). One can temporarily replace $H'$ 
by an affine patch $U$ embedded as a closed subset of an affine space $H$, 
and $X$ by $X \cap (U\times L)$, then apply the theorems; the resulting
statements then glue together to give the one for $X \subseteq H' \times L$.

\begin{Theorem}\label{thm:countdegen}
  Assume that $X \subseteq H\times L,\Pi,\Lambda,\barX,X'$ are as above.
%  (except that $H$ is allowed to be an arbitrary scheme).
  Assume one of the following:
  \begin{enumerate}
  \item $X$ is irreducible, $\Pi$ is normal, 
    the projection $X\to \Pi$ is degree $1$, and $\Lambda$ is reduced,
  \item $X'$ is reduced, or
  \item the fibers of $\barX \to \Pi$ are connected.
  \end{enumerate}
  (In fact (1) $\implies$ (2) $\implies$ (3) $\implies$ the projection
  $\barX\to \Pi$ is degree $1$ or $0$ to any component.)
  Let
  $$ \Lambda' := \left\{ h \in H : \{h\} \times L \subseteq X \right\}. $$
  Then $\Pi \supseteq \Lambda \supseteq \Lambda'$, and the
  image of the projection $X \to \Pi$ 
  is $\Pi \setminus \Lambda \union \Lambda'$.

  There is a decomposition 
  $X = (X\setminus (\Lambda'\times L)) \coprod (\Lambda'\times L)$,
  where $\pi$ gives an isomorphism of the first piece with 
  $\Pi\setminus\Lambda$.
  Consequently, 
  $$ [X'] = [X] + [\AA^1] \cdot [\Lambda \setminus \Lambda'] $$
  as elements of the Grothendieck ring of algebraic varieties. 
  In particular, if $X,H,L$ are defined over $\FF_p$,
  and $|A|$ denotes the number of $\FF_p$-rational points on $A$, then
  $$ |X'| = |X| + p\, \big|\Lambda \setminus \Lambda'\big| \equiv |X| \bmod p.$$

  There is a constructible injection $\iota: X \to X'$ defined by
  $$ \iota(h,\ell) := 
  \begin{cases}
    (h,\ell) &\text{if $h\in \Lambda'$} \\
    (h,0) &\text{if $h\notin \Lambda'$} 
  \end{cases}
  $$
  whose image is the complement of $(\Lambda\setminus \Lambda')\times L$.

  If $X'$ is reduced, then the proper map $\barX \to \Pi$ takes
  $[\calO_{\barX}] \mapsto [\calO_\Pi]$ as elements of $K$-theory
  (or even $G$-equivariant $K$-theory, if some group $G$ acts on $H$ and
  linearly on $L$, preserving $X \subseteq H\times L$ 
  and hence also $\Pi \subseteq H$).
\end{Theorem}

The $\dim H=1$ example of $X = \{h^2 \ell = 1\}$ with $p-1$ points,
degenerating to $X' = \{h^2 \ell = 0\}$ with $2p-1$ points, 
satisfies (3) and shows that $X'$ need not be reduced for this theorem.  
However, that is the condition we will make use of.
In section
\ref{sec:pfcountdegen} we give an example with $\Lambda' \neq \emptyset$.

\begin{Proposition}\label{prop:gvdsplit}
  Let $f \in \FF[h_1,\ldots,h_{n-1},\ell]$ be of degree $n$, of the
  form $f = \ell g_1 + g_2$ where $\ell\notdivides g_1,g_2$.
  Let $(h_i),\ell$ be the coordinates on $H,L$.
  If $\Tr_H(g_1^{p-1}\bullet)$ defines a Frobenius splitting on $H$,
  then $\Tr_{H\times L}(f^{p-1}\bullet)$ defines one on $H\times L$.
  Assume this hereafter.

  Let $X \subseteq H\times L$ be a subscheme compatibly split by
  $\Tr_H(g_1^{p-1}\bullet)$,
  and let $\Lambda' \subseteq \Lambda \subseteq \Pi \subseteq H$ be as above. 
  Then theorem \ref{thm:countdegen} applies.
  Moreover, 
  $\Pi$ and $\Lambda$ are compatibly split by $\Tr_H(g_1^{p-1}\bullet)$,
  though $\Lambda'$ may not be. 
\end{Proposition}

\begin{proof}
  The first statement follows from theorem \ref{thm:frobdegen} part (1).
  Its part (3) implies that $X'$
  is compatibly split by $\Tr_{H\times L}((\ell g_1)^{p-1}\bullet)$,
  so it is reduced, giving condition (2) of theorem \ref{thm:countdegen}.

  Being reduced, $X' = (\Pi\times \{0\}) \cup (\Lambda\times L)$,
  hence $\Pi\times \{0\}$, $\Lambda\times L$, 
  and $\Lambda\times \{0\} = (\Pi\times \{0\}) \cap (\Lambda\times L)$
  are compatibly split too. It follows that $\Pi,\Lambda$ are compatibly split 
  by $\Tr_{H}(g_1^{p-1}\bullet)$.
\end{proof}

The first example in section \ref{sec:pfcountdegen} is of this type,
and its $\Lambda'$ is not compatibly split in $H$. 

Applying proposition \ref{prop:gvdsplit} to $X = \{f=0\}$ itself,
we see the interrelation between
theorems \ref{thm:Trpointcount}, \ref{thm:frobdegen}, and \ref{thm:countdegen}.
Theorem \ref{thm:countdegen} says that $X$ and $X'$
have the same number of solutions, mod $p$.
Applying theorem \ref{thm:Trpointcount}, we see that 
$\Tr(f^{p-1}\bullet)$ defines a Frobenius splitting iff 
$\Tr(\init(f)^{p-1}\bullet)$ does. 
This gives an independent proof of theorem \ref{thm:frobdegen} part (1)
in this special situation.

\subsection{An important special case, which generalizes to schemes}
 %$\init(f) = \prod_i x_i$}

In this section we ask that $\init(f) = \prod_i x_i$. This condition
is a very restrictive one on degree $n$ polynomials; for example the
hypersurface defined by $f=0$ is necessarily singular. But there are
some important examples that have this property, and our results are
strongest here.

\begin{Theorem}\label{thm:initprod}
  Let $f \in \integers[x_1,\ldots,x_n]$ be a degree $k$ polynomial
  whose lexicographically first term is (a $\integers$-multiple of)
  a product of $k$ distinct variables.

  Let $Y$ be one of the schemes constructed from the hypersurface
  $f=0$ by taking components, intersecting, taking unions, and repeating.
  (Or more generally, let $Y$ be compatibly split with respect to the
  splitting $\Tr(f^{p-1}\bullet)$.)
  Then $Y$ is reduced over all but finitely many $p$, and over $\rationals$.

  Let $\lambda$ is the lexicographic weighting
  $(\varepsilon,\varepsilon^2,\ldots,\varepsilon^n)$ on the variables.
  Let $\init Y$ be the initial scheme of $Y$. Then (away from those $p$)
  $\init Y$ is a Stanley-Reisner scheme.
  There is a bijective constructible map $\AA^n$ to itself, 
  taking each $Y$ into its $\init Y$.
\end{Theorem}

We thank Bernd Sturmfels for his guess that $\init Y$ might
be a Stanley-Reisner ideal,
as a way of understanding corollary \ref{cor:allradical}
without direct reference to Frobenius splitting.
It would be interesting to know if $\init Y$ is reduced over $\integers$
not just $\rationals$, as holds \cite{KM,schubDegen} for the
examples in section \ref{sec:examples}.

It is tempting to pull back the standard paving of $\AA^n$ by tori
(one for each $S \subseteq \{1,\ldots,n\}$, defined by the equations
$\{x_i = 0$ iff $i\in S\}$) to try to get a paving of each $Y$, as in
\cite{Deodhar}. This works as long as each $\Lambda'$ occurring in
this succession of degenerations is compatibly split, but as
mentioned, they may not be.

Under this $\init f = \prod x_i$ condition, 
one can use lemma \ref{lem:coordspaces} and theorem \ref{thm:frobdegen}
to bound the number of $k$-dimensional compatibly split subvarieties
by $n\choose k$, as in \cite{Schwede2} (where they prove this bound without
assuming $\init f = \prod x_i$). If $f$ is homogeneous, 
then theorem \ref{thm:frobdegen} part (4) lets one show that $n\choose k$ also 
bounds the sum of the projective degrees of the $k$-dimensional 
compatibly split subvarieties.

In section \ref{sec:examples} we apply theorem \ref{thm:initprod}
to the general cases $n=2,n=3$, and to two specific stratifications;
the stratification of the space of matrices by matrix Schubert varieties,
and of opposite Bruhat cells by Kazhdan-Lusztig varieties.
To do this, we need the new result (theorem \ref{thm:kl})
that with respect to Bott-Samelson
coordinates on an opposite Bruhat cell, the complement of the big cell
is given by an equation $f$ with $\init f = \prod_i c_i$.

Part of this result has a generalization beyond affine space to
schemes, where it is closely related to a result of \cite{BK}:

\begin{Theorem}\label{thm:schemeversion}
  Let $X$ be a normal variety of dimension $n$, with 
  $\sigma \in H^0(X_{reg},\omega^{-1})$ a section of the anticanonical bundle
  over the regular locus $X_{reg}$. Let $x \in X_{reg}$ have local 
  coordinates $t_1,\ldots,t_n$, where the formal expansion of $\sigma$ at $x$ is
  $$ \sigma = f(t_1,\ldots,t_n)\ (dt_1\wedge \cdots \wedge dt_n)^{-1}. $$
  \begin{enumerate}
  \item (From the proof of \cite[proposition 1.3.11]{BK}.) \\
    If $X$ is complete, and
    the unique lowest-order term of $\sigma$ is $\prod_{i=1}^n t_i$,
    then there exists a unique Frobenius splitting of $X$ that
    compatibly splits the divisor $\{\sigma=0\}$.
    In particular, if $\{\sigma=0\}$ has $n$ components
    smooth at $x$ and meeting transversely there, the coordinates $\{t_i\}$
    can be chosen to ensure this condition on $\sigma$.
  \item If the initial term of $\sigma$ is $\prod_{i=1}^n t_i$ for
    some term order, then there exists a Frobenius splitting of $X$ that
    compatibly splits the divisor $\{\sigma=0\}$.
    If $X$ is complete, then the splitting is unique.
  \end{enumerate}
\end{Theorem}

In proposition \ref{prop:multfreesplit}
we give an application of this to Brion's ``multiplicity-free subvarieties
of $G/B$'': if $X$ is a multiplicity-free {\em divisor}, then $G/B$ possesses
a Frobenius splitting compatibly splitting $X$. 

\subsection{Application to Gr\"obner bases}

In a finite poset $P$, call an element $p$ \defn{basic} if $p$ is {\em not}
the unique greatest lower bound of $\{q \in P : q > p\}$.
It is then trivial to prove (\cite{biGrassmannian,GeckKim}, where they also
determine the basic elements in Bruhat orders) that any $p$
{\em is} the greatest lower bound of $\{q \in P : q \geq p, q$ basic$\}$. 

\begin{Theorem}\label{thm:Grobner}
  Let $f \in \integers[x_1,\ldots,x_n]$ be a degree $n$ polynomial
  with $\init f = \prod_{i=1}^n x_i$.
  Let $\YY_f$ be a poset of compatibly split subvarieties
  with respect to the splitting $\Tr(f^{p-1}\bullet)$ (here $p$ varies
  over some infinite set of primes), ordered by inclusion. 
  (It need not be all of them.) Then over the rationals:

  \begin{enumerate}
  \item Any $Y\in \YY_f$ has a Gr\"obner basis $(g_i)$ over whose 
    initial terms $(\init g_i)$ are squarefree monomials.
  \item Any $Y\in \YY_f$ is the scheme-theoretic intersection of
    $\{Z \in \YY_f : Z \geq Y, Z$ basic in $\YY_f\}$, and of course
    it suffices to use only the minimal elements.
  \item If we concatenate Gr\"obner bases of the minimal
    $\{Z \in \YY_f : Z \geq Y, Z$ basic in $\YY_f\}$, 
    we get a Gr\"obner basis of $Y$.
  \end{enumerate}

  Indeed (2) holds for $\YY$ a set of compatibly split subvarieties
  in any split scheme, and (3) holds whenever $\Tr((\init f)^{p-1}\bullet)$
  is a splitting.
\end{Theorem}

As any single polynomial forms a Gr\"obner basis, we see that
a concatenation of Gr\"obner bases is usually {\em not} a Gr\"obner basis.
The special geometry of our situation is explained in lemma \ref{lem:Grobner}.
In section \ref{ssec:matrixSchubs} we use theorem \ref{thm:Grobner} to
recover the main results of \cite{Fulton92,KM}.

\subsection{Acknowledgements}
We thank Michel Brion, Brian Conrad, Nick Katz, Michael Larsen, Jiang-Hua Lu,
Gregg Musiker, Karl Schwede, David Speyer, Mike Stillman, Bernd Sturmfels, 
and Terry Tao for useful conversations, some of which occurred at the 2009 MSRI
Program in Algebraic Geometry.  Many of the calculations were done
with Macaulay 2 \cite{M2}.

\section{Near-splittings of affine space and its ideals}
\label{sec:nearsplittings}

Following \cite[Section 1.3.1]{BK}, we describe 
all the near-splittings on $\FF[x_1,\ldots,x_n]$, where $\FF$ is
a perfect field over $\FF_p$. (In \cite{BK} they assume in general 
that $\FF$ is algebraically closed, but make no use of this in that section.)

\junk{
\begin{Lemma}\label{lem:checknearsplitting}
  Let $S$ contain a perfect field $\FF$ over $\FF_p$, and $\varphi:S \to S$ be 
  an additive map such that $\varphi(r^p b) = r\, \varphi(b)$ for all $r\in\FF$.
  Then to check whether $\varphi$ is a near-splitting of $S$, 
  it is enough to check whether $\varphi(a^p b) = a\, \varphi(b)$ 
  for $a,b$ in a spanning set of the $\FF$-module $S$.
\end{Lemma}
\begin{proof}
  If $a = \sum_i r_i s_i, b = \sum_j r'_j s_j$ for $\{s_i\}$ the spanning set
  and $\{r_i\}$ the $\FF$-coefficients, then
  $$ \varphi\left(a^p b\right)
  = \varphi\left(\left(\sum_i r_i s_i\right)^p \sum_j r'_j s_j\right)
  = \varphi\left(\sum_i r_i^p s_i^p \sum_j r'_j s_j\right)
  = \sum_{i,j} \varphi\left(r_i^p s_i^p r'_j s_j\right)
  $$
  $$
  = \sum_{i,j} r_i \sqrt[p]{r'_j} \ \varphi\left(s_i^p s_j\right)
  = \sum_{i,j} r_i \sqrt[p]{r'_j}\ s_i\ \varphi \left(s_j\right)
  = \sum_{i,j} r_i  s_i \ \varphi\left(r'_j s_j\right)
  = \sum_i r_i s_i \ \varphi\left(\sum_j r'_j s_j\right)
  = a \varphi(b). $$  
\end{proof}
}

\begin{Proposition}\label{prop:nearsplittings}\cite[section 1.3.1]{BK}
  Let $\FF$ be a perfect field over $\FF_p$. Then there exists a unique
  near-splitting $\Tr(\bullet)$ on $\FF[x_1,\ldots,x_n]$ such that
  for each monomial $m = \prod_i x_i^{e_i}$, 
  $$ \Tr(m) = 
  \begin{cases}
    \sqrt[p]{m \prod_i x_i} \bigg/ \prod_i x_i 
    & \text{if $m \prod_i x_i$ is a $p$th power} \\
    0 & \text{otherwise.}
  \end{cases} $$
  For each $f\in \FF[x_1,\ldots,x_n]$, 
  the map $\Tr(f\bullet) : g \mapsto \Tr(fg)$ is a near-splitting,
  and the association $f \mapsto \Tr(f\bullet)$ is a bijection 
  from $\FF[x_1,\ldots,x_n]$ to the set of near-splittings.
\end{Proposition}

Hereafter we will assume that $\FF$ is a perfect field over $\FF_p$,
and the near-splittings we will consider will all be of the form
$c \Tr(f^{p-1}\bullet)$ for some $f \in \FF[x_1,\ldots,x_n]$ and $c\in \FF$.

\begin{Lemma}\label{lem:splitsf} 
  Let $f \in \FF[x_1,\ldots,x_n]$. 
  \begin{enumerate}
  \item If $\Tr(f^{p-1})$ is a unit, and $c$ is its inverse,
    then $c\, \Tr(f^{p-1})$ is a splitting of $\FF[x_1,\ldots,x_n]$.
  \item 
    The principal ideal
    $\langle f\rangle$ is compatibly near-split by $\Tr(f^{p-1}\bullet)$,
    and compatibly split by $c\, \Tr(f^{p-1}\bullet)$ from part (1) 
    if $c\, \Tr(f^{p-1}) = 1$.
  \end{enumerate}
\end{Lemma}

\begin{proof}
  \begin{enumerate}
  \item $c\, \Tr(f^{p-1} \bullet) = \Tr(c^p f^{p-1} \bullet)$, 
    hence is a near-splitting, and the remaining condition
    that $c\, \Tr(f^{p-1} \bullet) = 1$ is how we chose $c$.
  \item 
    If $r f \in \langle f\rangle$, then
    $\Tr(f^{p-1} rf) = f\, \Tr(r) \in \langle f\rangle$; likewise
    $c\, \Tr(f^{p-1} rf) = f c\, \Tr(r) \in \langle f\rangle$.
  \end{enumerate}
\end{proof}

We were tempted to generalize the definition of splitting by allowing
$\phi(1)$ to be a unit rather than actually $1$. This would make some
theorems nicer to state, but did not seem worth the confusion to 
people familiar with the usual definition.

We will later be interested in near-splittings $\Tr(f \bullet)$
on $R[x_1,\ldots,x_n]$ where $R$ is a certain \defn{perfect ring} over $\FF_p$,
meaning that the Frobenius map $R\to R$ is bijective. (We won't need to 
generalize the results of this section, though their proofs from 
\cite[section 1.3.1]{BK} go through without change.)
It is easy to show that such rings are Noetherian only when they are fields.

\junk{
of the base ring is bijective.\footnote{%
  Unfortunately such rings are non-Noetherian, when they are not fields, 
  as we now show. First replace $R$ by its quotient by 
  a prime component of $0$; the result will again be split.
  Let $b \in R$, $b\neq 0$, and consider the
  increasing chain of principal ideals 
  $\langle b \rangle \leq \langle \varphi(b) \rangle 
  \leq \langle \varphi(\varphi(b)) \rangle \ldots$.
  If $R$ is Noetherian making this terminate, then for some 
  $c = \varphi(\varphi(\cdots(b)\cdots)$, one has $\varphi(c) = rc$, 
  so $c = r^p c^p$. Since we are in a domain, $1 = (r^p c^{p-2}) c$,
  so $c$ and hence $b$ are units.} 
This generalizes the more usual notion of perfect field;
%(e.g. algebraically closed fields).
one non-field example is the ring $\FF_p[t^r : r\in\rationals, r > 0]$.
}

\section{Proof of theorem \ref{thm:Trpointcount}}

\begin{Lemma}\label{lem:degbound}
  Let $f \in \FF[x_1,\ldots,x_n]$ be of degree at most $n$ over a perfect
  field $\FF$. Then $\Tr(f^{p-1})$ 
  is the $p$th root of the coefficient on $\prod_i x_i^{p-1}$ in $f^{p-1}$.
\end{Lemma}

\begin{proof}
  The only monomials noticed by $\Tr(\bullet)$ 
  are of the form $m^p \prod_i x_i^{p-1}$, where $m$ is itself a monomial,
  and so have degree $p \deg m + (p-1) n$.
  By the assumption on $\deg f$, its power $f^{p-1}$ can't contain 
  such a monomial other than the one for $m=1$.
\end{proof}

For $n=3$, the following is a standard argument from the theory
of supersingular elliptic curves, and was studied for hypersurfaces
in \cite[example 2.3.7.17]{ellipticCurves}.

\junk{
\begin{Theorem}{thm:pointcount}
  Let $f \in \integers[x_1,\ldots,x_n]$, $n>0$, with $f$ of degree at most $n$.
  Then the number of $\vec v \in \FF_p^n$ satisfying $f(\vec v)=0$
  is congruent to $(-1)^{n-1} \Tr(f^{p-1})$ mod $p$. 

  In particular, if the number of $\vec v \in \FF_p^n$ with $f(\vec v)=0$
  is not a multiple of $p$, then some multiple of $\Tr(f^{p-1})$
  defines a splitting on affine space with respect to which 
  $\langle f \rangle$ is compatibly split.
\end{Theorem}
}

\begin{proof}[Proof of theorem \ref{thm:Trpointcount}, when $f$ doesn't factor.]
  First observe that for any $a \in \FF_p$, 
  $$ 1 - a^{p-1} \equiv
  \begin{cases}
    1 &\text{if $a=0$} \\
    0 &\text{if $a\neq 0$}
  \end{cases} $$
  by Fermat's Little Theorem. 
  If we let $f^{p-1} = \sum_{\vec e} c_e \prod_{i=1}^n x_i^{e_i}$, then
  $$
  \# \left\{\vec v \in \FF_p^n : f(\vec v)=0 \right\}
  = \sum_{\vec v \in \FF_p^n} (1 - f(\vec v)^{p-1}) 
  = p^n - \sum_{\vec v \in \FF_p^n} f(\vec v)^{p-1} 
  \equiv - \sum_{\vec v \in \FF_p^n} f(\vec v)^{p-1} \bmod p
  % \qquad \text{since $n>0$} 
  $$
  $$
  = -\sum_{\vec v\in\FF_p^n} \sum_{\vec e} c_e \prod_{i=1}^n v_i^{e_i}
  = - \sum_{\vec e} c_e \sum_{\vec v\in\FF_p^n} \prod_{i=1}^n v_i^{e_i}
  = - \sum_{\vec e} c_e \prod_{i=1}^n \sum_{v_i \in \FF_p} v_i^{e_i}.
  $$
  Now consider the sum $\sum_{v \in \FF_p} v^e$. If $e=0$, this is 
  $p\cdot 1 \equiv 0$. 
  If $(p-1)|e$ and $e>0$, this is $0+(p-1)\cdot 1 \equiv -1$.
  Otherwise let $b$ be a generator of $\FF_p^\times$ so $b^e\neq 1$, 
  and observe that $\sum_{v \in \FF_p} (bv)^e = b^e \sum_{v \in \FF_p} v^e$ 
  is just a rearrangement of $\sum_{v \in \FF_p} v^e$, 
  so $\sum_{v \in \FF_p} v^e = 0$.  Omitting zero terms from the sum, we have
  $$
  \# \left\{\vec v \in \FF_p^n : f(\vec v)=0 \right\}
  \equiv - \sum_{\vec e:\ \forall i, e_i>0, (p-1)|e_i} c_e 
  \prod_{i=1}^n (-1) $$
  At this point the only terms entering have each $e_i \geq p-1$,
  so $\sum_i e_i \geq (p-1)n$. But that is $\geq \deg (f^{p-1})$.
  So (as in the proof of lemma \ref{lem:degbound})
  the only $\vec e$ has $e_i = p-1\ \forall i$.
  $$
  \# \left\{\vec v \in \FF_p^n : f(\vec v)=0 \right\}
  \equiv - (-1)^n c_{p-1,p-1,\ldots,p-1}. $$
  On the other hand,
  $$ \Tr(f^{p-1}) 
  = \Tr\left(\sum_{\vec e} c_e \prod_{i=1}^n x_i^{e_i}\right)
  = \sum_{\vec e} \Tr\left(c_e \prod_{i=1}^n x_i^{e_i}\right)
  = \sum_{\vec e} c_e \Tr\left(\prod_{i=1}^n x_i^{e_i}\right) $$
  where the last step uses the fact that each $c_e$ is in the prime field.
  The term $\Tr(\prod_{i=1}^n x_i^{e_i})$ is $0$ unless each $e_i\geq p-1$,
  so degree-counting as before, the only term that survives is
  $c_{p-1,p-1,\ldots,p-1}$. Combining, we get
  $\# \left\{\vec v \in \FF_p^n : f(\vec v)=0 \right\}
  \equiv (-1)^{n-1} \Tr(f^{p-1})$.
\end{proof}

One can see from the proof that if $\deg f < n$, then $\Tr(f^{p-1}) = 0$,
so $\#\{\vec v : f(\vec v)=0\}$ is a multiple of $p$.
This can be generalized (via a very similar proof) as follows:

%We recall a theorem about a closely related situation.

\begin{Theorem*}[Chevalley-Warning]
  Let $\{f_i\}$ be a set of polynomials in $\FF[x_1,\ldots,x_n]$,
  $\FF$ a finite field of characteristic $p$, 
  such that $\deg (\prod_i f_i) < n$.
  Then $\# \{\vec v \in \FF^n : f_i(\vec v) = 0\ \forall i\}$
  is a multiple of $p$.
\end{Theorem*}

To have Frobenius splittings, we want $\Tr(f^{p-1}) \neq 0$, so we
will want our polynomial $f$ to have degree $n$. 
The Chevalley-Warning theorem is useful to us all the same, 
as in examples we will often want $f$ to factor.

\begin{proof}[Proof of the remainder of theorem \ref{thm:Trpointcount}:
  when $f$ factors.]
  Let $X_i = \{\vec v : f_i(\vec v) = 0\}$, so the left side is
  $    \#\left( \bigcup_i X_i \right)$.
  Inclusion-exclusion says
  $$ \# \left( \bigcup_i X_i \right)
  = \sum_{S \subseteq \{1,\ldots,m\}, S \neq \emptyset} 
  (-1)^{|S|-1} \# \left( \bigcap_{i\in S} X_i \right)
  = \sum_{S \neq \emptyset} (-1)^{|S|-1} \# \left\{\vec v \in \FF^n 
  : \forall i\in S, f_i(\vec v) = 0\right\}. $$
  For each proper subset $S \subsetneq \{1,\ldots,m\}$, 
  and by the assumption that the $\{f_i\}$ are nonconstant,
  the Chevalley-Warning theorem applies to 
  $\{\vec v \in \FF^n : \forall i\in S, f_i(\vec v) = 0\}$. So
  mod $p$, only the $S = \{1,\ldots,m\}$ term survives:
%  \begin{align*}

\hfill{$
    \#\left( \bigcup_i X_i \right)
%    \sum_{S = \{1,\ldots,m\}} (-1)^{|S|-1} \# \{\vec v \in \FF^n 
%    : \forall i\in S, f_i(\vec v) = 0\} &\bmod p \\
    \equiv
       (-1)^{m-1} \# \{\vec v \in \FF^n  : \forall i, f_i(\vec v) = 0\}.$
}\hfill
%  \end{align*}
\end{proof}

It is also easy to see from the proof of theorem \ref{thm:Trpointcount}
that there is no easy relation between point-counting and $\Tr(f^{p-1})$
if $\deg f > n$. The point count draws focus on exponents $e_i > 0$ with
$(p-1) | e_i$, whereas $\Tr(f^{p-1})$ is concerned with exponents $e_i$
with $e_i \equiv p-1 \bmod p$. These match up well only if degree 
considerations force $e_i = p-1$. 

Put another way, if $\deg f > n$ and $\Tr(f^{p-1})$ is a unit (which
can only happen if $f$ is inhomogeneous), then the splendid geometric
consequences of Frobenius splitting hold but are not detected by point-counting.

\section{Proofs of lemma % \ref{lem:coordspaces} and
  \ref{lem:Trinit} and theorem \ref{thm:frobdegen}}
\label{sec:Trinit}
%  and corollary \ref{cor:monomial}}

\begin{proof}[Proof of theorem \ref{thm:frobdegen}, part (1)]
  If $\prod_i x_i$ is not in $f$'s Newton polytope, then 
  $\prod_i x_i^{p-1}$ is not in $f^{p-1}$'s Newton polytope. 
  Hence by lemma \ref{lem:degbound}, $\Tr(f^{p-1}) = 0$.

  \junk{
    For any polynomial $g$, write
    $g = \sum_m g_m$ where $g_m$ is the sum of $g$'s terms 
    $c \prod_i x_i^{e_i}$ having $\sum_i \lambda_i e_i = m$.
    Then $(f^{p-1})_m = \sum_{m_1,m}$
    
    Write $f = \sum_m f_m$ where $f_m$ is the sum of $f$'s terms 
    $c \prod_i x_i^{e_i}$ having $\sum_i \lambda_i e_i = m$.
    By definition, the maximum $m$ with $f_m \neq 0$ has $f_m = \init(f)$. 
  }
  Since $\prod_i x_i$ lies in $\init(f)$, we know
  $\prod_i x_i^{p-1}$ lies in $\init(f)^{p-1} = \init(f^{p-1})$.
  So the coefficient on $\prod_i x_i^{p-1}$ in $f^{p-1}$ is the
  same as its coefficient in $\init(f)^{p-1}$.
  Now apply lemma \ref{lem:degbound} to infer
  $\Tr(f^{p-1}) = \Tr(\init(f)^{p-1})$.
\end{proof}

\junk{

\begin{Theorem}{thm:criterion}
  Let $f \in \integers[x_1,\ldots,x_n]$ be of degree $n$.
  Assume that for some positive weighting of the variables, 
  the unique term of highest weight is $c \prod_i x_i$.
  Then for each $p\notdivides c$, the near-splitting $\Tr(f^{p-1})$ is a 
  Frobenius splitting, which compatibly splits the ideal $\langle f\rangle$.

  If $f$ factors as $\prod_j f_j$, then the condition is equivalent
  to the following: the highest-weight term $\init(f_j)$ of each $f_j$ 
  is unique and a squarefree monomial, and $gcd(\init(f_i),\init(f_j))=1$
  for $i\neq j$.  
\end{Theorem}

\begin{proof}
  If $f = c \prod_i x_i + R$ where each term in $R$ has strictly higher weight,
  then $f^{p-1} = c^{p-1} \prod_i x_i^{p-1} + R'$ where each term in $R'$
  is a product of $p-1$ terms of $f$, all of which have weight at least
  equal to and some strictly greater than that of $c \prod_i x_i$.
  Hence the coefficient of $\prod_i x_i^{p-1}$ in $f^{p-1}$ is $c^{p-1}$.

  Since $\deg f^{p-1} = (p-1)n$, no other term $\prod_i x_i^{pk_i + p-1}$ 
  can appear in $f^{p-1}$. Together, one sees that $\Tr(f^{p-1}) = c^{p-1}$.
  If $p\notdivides c$, this is $1$ by Fermat's Little Theorem.
  Apply lemma \ref{lem:splitsf}.
\end{proof}

There is a simple generalization, where $f$ has more than one term
of highest weight, and $f_0$ is the sum of those terms.
Then $\Tr(f^{p-1}) = \Tr(f_0^{p-1})$, and hence they both define
a splitting or neither does. We consider this situation 
further in the next section.
}

%We finally make use of a general perfect base ring.

To study the degeneration, it will be convenient to introduce the 
perfect base ring $R = \FF_p[t^r : r\in\rationals_+]$ of Puiseux polynomials.
Define the ring endomorphism $h_\lambda$ of $R[x_1,\ldots,x_n]$ by
$$ h_\lambda(x_i) = x_i t^{\lambda_i}, \qquad \text{so} \qquad
(h_\lambda\cdot g)(x_1,\ldots,x_n) 
= g(x_1 t^{\lambda_1},\ldots,x_n t^{\lambda_n}).
$$
(where $h$ is for ``homogenize'').
Then $\init(g)$ is the part of $h_\lambda(g)$ with highest $t$-degree.%
\footnote{Perhaps it would be more natural to take 
  $x_i \mapsto x_i t^{-\lambda_i}$, and clear denominators by 
  multiplying by some $t^M$, so
  $\init(g)$ would be $(t^M h_\lambda(g))|_{t=0}$.
  It didn't seem to be worth keeping track of an extra sign, however.}
Note that $t$ is considered part of the base ring and {\em not} a
new variable, for purposes of defining $\Tr(\bullet)$ on
$R[x_1,\ldots,x_n]$; in particular $\Tr(f(t) g) = f(t^{1/p}) \Tr(g)$
for any $f(t)\in R$. 

\begin{proof}[Proof of lemma \ref{lem:Trinit}.]
  First we prove
  $$ \Tr(h_\lambda(g)) = h_\lambda(Tr(g)) t^{\frac{p-1}{p}\sum_i \lambda_i}$$
  for $g \in \FF_p[x_1,\ldots,x_n]$, i.e. when $g$ has no $t$-dependence.
  Both sides are additive, so it is enough to check 
  for $g = c \prod_i x_i^{e_i}, c\in \FF_p$. 
  If each $e_i \equiv - 1 \bmod p$, then
  \begin{align*}
    \Tr\left(h_\lambda(c \prod_i x_i^{e_i})\right)
  &= \Tr\left(c \prod_i (x_i t^{\lambda_i})^{e_i}\right) 
  = \Tr\left(t^{\sum_i \lambda_i e_i} c \prod_i x_i^{e_i}\right) 
  = t^{\frac{1}{p}\sum_i \lambda_i e_i} c \Tr\left(\prod_i x_i^{e_i}\right) \\
  &= t^{\frac{1}{p}\sum_i \lambda_i e_i} c \prod_i x_i^{(e_i+1)/p-1} 
  %  \\ &
  = t^{\frac{p-1}{p} \sum_i \lambda_i} c \prod_i (x_i t^{\lambda_i})^{(e_i+1)/p-1} \\
  &= t^{\frac{p-1}{p} \sum_i \lambda_i} 
  h_\lambda\left(c \prod_i x_i^{(e_i+1)/p-1}\right) 
  = t^{\frac{p-1}{p} \sum_i \lambda_i} 
  h_\lambda\left(\Tr\left(c \prod_i x_i^{e_i}\right)\right)
  \end{align*}
  and both sides are zero otherwise. This proves the equation.

  Now let $g$ be general, and consider $g$'s Newton polytope $P$.
  $\Tr(\bullet)$ and $\init$ are sensitive to different parts of $g$'s
  Newton polytope: $\Tr(g)$ only depends on the terms lying on the 
  intersection of $P$ with a coset $C$ of a lattice 
  (namely, where all exponents are $\equiv -1 \bmod p$), whereas $\init(g)$
  only depends on the terms lying over one face $F$ of $P$. 

  There are then two cases. If some terms of $g$ lie over $F\cap C$,
  then we can pick out the terms lying over $F$, and from those pick out the
  terms also lying over $C$, or do so in the opposite order.
  Either way we pick up the terms lying over $F\cap C$, and apply 
  $\Tr(\bullet)$ to them, obtaining $\Tr(\init(g)) = \init(\Tr(g))$.

  The other possibility is that no terms lie over $F\cap C$ (e.g. if
  $F\cap C = \nulset$). Then $\init(g)$ picks out the terms lying
  over $F$, and $\Tr(\init(g)) = 0$.
\end{proof}

\begin{proof}[Proof of part (2)]
  The ideal $\init(I)$ is linearly generated by $\{\init(g) : g\in I\}$.
  By lemma \ref{lem:Trinit},
  $$ \Tr\left(\init(f)^{p-1} \init(g)\right)
  = \Tr\left(\init(f^{p-1} g)\right)
  = \bigg(\init \Tr(f^{p-1} g) \text{ or } 0\bigg)
  \in \init(I) $$
  so $\init(I)$ is compatibly near-split by $\Tr(\init(f)^{p-1}\bullet)$.
\end{proof}

\junk{
\begin{Proposition}\label{prop:degennearsplit}
  Let $f \in \FF_p[x_1,\ldots,x_n]$, let $\lambda$ be a weighting
  on the variables, and let $I$ be an ideal.
  If $I$ is compatibly near-split by $\Tr(f^{p-1}\bullet)$,
  then $\lt(I)$ is compatibly near-split by $\Tr(\lt(f)^{p-1}\bullet)$.  
\end{Proposition}

\begin{proof}
  Extend scalars to the perfect ring $R = \FF_p[t^r : r \in \rationals_+]$.
  First we claim that $R\tensor I$ is compatibly near-split by
  the near-splitting $\Tr(f^{p-1}\bullet)$ on $R[x_1,\ldots,x_n]$.

  Use $\lambda$ to define an $R$-linear endomorphism $w_\lambda$ of 
  $R[x_1,\ldots,x_n]$ by $w_\lambda(x_i) = x_i t^{\lambda_i}$.
  Let $I'$ be the ideal generated by $w_\lambda(R \tensor I)$, 
  and $I''$ its saturation $I' : \langle t^\infty \rangle$.

  {show $I'$ compatibly near-split, hence $I''$ is too;
    that proof didn't require splitting. Need to show $I''$ is
    flat over $\Spec R$. 
    Then the fiber over $\langle t^r : r\in \rationals_+\rangle$
    is $\lt(I)$.  }
\end{proof}
\begin{proof}[Proof of part (2)]
  Apply proposition \ref{prop:degennearsplit} in the case that
  $\Tr(\lt(f)^{p-1}\bullet)$ defines a splitting.
\end{proof}
}

We give now a criterion which may be of independent interest,
guaranteeing that the limit of an intersection is the
entire intersection of the limits.

\begin{Lemma}\label{lem:Grobner}
  Let $A$ be a discrete valuation ring with parameter $t$, 
  so $S = \Spec A$ has one open point $S^\times$ and one closed point
  $S_0$, and let $F$ be a flat family over $\Spec A$.
  Let $X,Y$ be two reduced flat subfamilies, and assume that the special fiber
  $(X\union Y)_0$ of their union is reduced.

  Then $X \cap Y$ is the closure of $X^\times \cap Y^\times$; it has no
  components lying entirely in the special fiber. In particular
  $\left(\overline{X^\times \cap Y^\times}\right)_0 = X_0 \cap Y_0$.
\end{Lemma}

\begin{proof}
  Consider two gluings of $X$ to $Y$, along their
  common subschemes $\overline{X^\times \cap Y^\times} \into X\cap Y$:
  $$    (X \coprod Y) \bigg/ \overline{X^\times \cap Y^\times}
  \quad\onto\quad (X \coprod Y) \bigg/ (X \cap Y) 
  \quad\iso\quad X \union Y \subseteq F. $$
  Call this map $\pi: Z_1 \onto Z_2$.
  It is finite, and an isomorphism away from $t=0$, and $Z_1,Z_2$ are reduced,
  so $Fun(Z_1)$ is integral inside $Fun(Z_2)[t^{-1}]$. 

  If $\overline{X^\times \cap Y^\times} \neq X\cap Y$,
  so $Fun(Z_1) \neq Fun(Z_2)$, then there exists $r\in Fun(Z_2)$ 
  such that $r/t \in Fun(Z_1) \setminus Fun(Z_2)$. By the integrality,
  $r/t$ satisfies a monic polynomial of degree $m$
  with coefficients in $Fun(Z_2)$. Hence $r^m \equiv 0 \bmod t$,
  but $r \not\equiv 0 \bmod t$ (since $r/t \notin Fun(Z_2)$),
  so $Fun(Z_2)/\langle t \rangle = Fun((X\union Y)_0)$ has nilpotents,
  contrary to assumption.
\end{proof}

Using the {\em branchvariety} framework of \cite{AK} (from whose lemma
2.1(1) this proof has been copied), one can analyze the situation when
the ``limit scheme'' $(X\cup Y)_0$ of $X^\times \cap Y^\times$ is not
assumed reduced.
Then the zero fiber of $(X \coprod Y) \big/ \overline{X^\times \cap Y^\times}$
is the ``limit branchvariety'' of $X^\times \cap Y^\times$, which
maps to the limit scheme $(X\cup Y)_0$. In \cite{AK} we prove that
limit branchvarieties are unique, hence this map is an isomorphism
iff $(X\cup Y)_0$ is reduced.

\begin{Corollary}\label{cor:Grobner}
  Let $S = \{I\}$ be a finite set of polynomial ideals in 
  $\FF[x_1,\ldots,x_n]$ such that for any 
  $S' \subseteq S$, $\init \bigcap_{I \in S'} I$ is radical. 
  Then $\init \left( \sum_{I\in S} I \right) = \sum_{I\in S} \init I$. 

  Put another way, 
  for each $I\in S$, let $G_I$ be a Gr\"obner basis for $I$.
  Then $\Union_{I\in S} G_I$ is a Gr\"obner basis for $\sum_{I\in S} I$.
\end{Corollary}

\begin{proof}
  For any $S' \subseteq S$, the conditions apply to $S'$, so using induction
  we can reduce to the case $S=\{I_1,I_2\}$. 

  Let $F$ be the trivial family over $\FF[[t]]$ with fiber $\AA^n$, 
  and $X,Y \subseteq F$ be the Gr\"obner families whose general fibers
  are defined by $I_1,I_2$ and special fibers by $\init I_1, \init I_2$.
  Then $\init(I_1 + I_2)$ and $\init I_1 + \init I_2$ are
  the defining ideals of 
  $\left(\overline{X^\times \cap Y^\times}\right)_0$ and $X_0 \cap Y_0$
  respectively.
  The condition $\init (I_1 \cap I_2)$ radical allows us to invoke
  lemma \ref{lem:Grobner} to infer these are equal.
\end{proof}

A very similar result appears in \cite[lemma 3.2]{tropical}.

\begin{proof}[Proof of part (3)]
  Plainly any $Y' \in \YY_{\init f}$ has some $Y \in \YY_f$ 
  such that $\init Y \supseteq Y'$; take $Y = \AA^n$. 
  We first need to show there is a unique minimal such $Y$.
  By corollary \ref{cor:Grobner}, if $Y' \subseteq \init Y_1$ and
  $Y' \subseteq \init Y_2$, then
  $Y' \subseteq \init (Y_1 \cap Y_2) = \init \Union_Z Z = \Union_Z \init Z$
  where $Z$ ranges over the components (all compatibly split) 
  of the $Y_1 \cap Y_2$. Hence $Y' \subseteq \init Z$ for one of those $Z$.
  If $Y_1,Y_2$ were both minimal, then $Z = Y_1$ and $Z = Y_2$,
  showing the uniqueness.

  Given $Y \in \YY_f$, let $Y'$ be a component of $\init Y$, 
  necessarily in $\YY_{\init f}$ by part (2). Since $\dim \init Y = \dim Y$,
  there can be no $Z \subsetneq Y, Z \in \YY_f$ with $Y' \subseteq \init Z$.
  Therefore the map takes $Y' \mapsto Y$, proving the surjectivity.

  Now $Y_1 = \pi_{f,\init}(Y_1')$, and $Y_2 \in \YY_f$.
  \begin{itemize}
  \item $\Longrightarrow$: 
    If $Y_1 \supseteq Y_2$, then $Y_1' \subseteq \init Y_1 \subseteq \init Y_2$,
    so $Y_1'$ is contained in some component $Y_2'$ of $\init Y_2$.
    Then $Y_2' \subseteq \init Y_2$. 

    Since $Y_2$ is irreducible,
    $\init Y_2$ is equidimensional of the same dimension as $Y_2$,
    so $\dim Y_2' = \dim Y_2$, and also any $Z \subsetneq Y_2$
    has $\dim Z < \dim Y_2$. Hence $Y_2' \not\subset \init Z$.
    Together, this shows $\pi_{f,Y_2'} = Y_2$.
  \item $\Longleftarrow$:
    Now say $\exists Y_2' \supseteq Y_1'$.
    Then $Y_1' \subseteq Y_2' \subseteq \init \pi_{f,\init}(Y_2')$.
    The definition of $\pi_{f,\init}(Y_1')$ is as the {\em least} $Y\in \YY_f$
    such that $Y_1' \subseteq \init Y$, hence 
    $\pi_{f,\init}(Y_1') \subseteq \pi_{f,\init}(Y_2')$.
    If $\pi_{f,Y_2'} = Y_2$, this says $Y_1 \subseteq Y_2$.
  \end{itemize}
\end{proof}

We conjecture that the equivalence relation induced on $\YY_{\init f}$
by $\pi_{f,\init}$ can be determined from the $\naturals$-valued
function $Y' \mapsto \dim \pi_{f,\init}(Y')$, as the symmetric,
transitive extension of the relation ``$Y'_1 \subseteq Y'_2$ 
  and $\dim \pi_{f,\init}(Y'_1) = \dim \pi_{f,\init}(Y'_2)$''. 

\begin{proof}[Proof of part (4)]
  First we claim that $C$ is a component of $\init Y$
  iff $\dim C = \dim Y$ and $\pi_{f,\init}(C) = Y$ and $\dim C = \dim Y$.
  \begin{itemize}
  \item $\Longrightarrow$: $\init Y$ is equidimensional of dimension $\dim Y$,
    so $C$ has that same dimension. Also $\pi_{f,\init}(C) \subseteq Y$,
    so $C \subseteq \init \pi_{f,\init}(C) \subseteq \init Y$, hence
    $\pi_{f,\init}(C)$ is contained in $Y$ and of the same dimension.
    Since $Y$ is irreducible, $\pi_{f,\init}(C) = Y$.
  \item $\Longleftarrow$: $C \subseteq \init Y$, whose components have
    dimension $\dim Y$, one of whom contains $C$. Hence $C$ equals
    that component, as above.
  \end{itemize}
  Since $f$ is homogeneous, by lemma \ref{lem:homog}, each $Y \in \YY_f$
  is an affine cone and has a well-defined $\deg Y$. Then
  $$ \deg Y = \deg (\init Y) 
  = \deg \sum_{C \in comps(\init Y)} \deg C 
  = \deg \sum_{C : \pi_{f,\init}(C) = Y, \dim C = \dim Y} \deg C $$
  with the last equality by the claim above.
\end{proof}

One can extend this to a calculation of the Hilbert series, not just
the degree, using the result of \cite{Mobius}.

\section{Proof of theorem \ref{thm:countdegen}}
\label{sec:pfcountdegen}
 
\begin{proof}[Proof of theorem \ref{thm:countdegen}]
  \emph{(1)$\implies$(2).} 
  This is the Geometric Vertex Decomposition Lemma of \cite{automatically}.

  \emph{(2)$\implies$(3).}
  Let $I_X$ be the ideal defining $X \subseteq H\times L$.
  Working through the definition of the scheme $X'$, we see its
  ideal of definition is $\init(I)$ with respect to 
  the weighting on the variables $\lambda(h_i)=0, \lambda(\ell)=1$. 

  Let $I_\Pi \leq I_\Lambda$ be the ideals defining
  $\Pi,\Lambda\subseteq H$, with generators $I_\pi = \langle (p_i) \rangle$, 
  $I_\Lambda = \langle (p_i), (q_i) \rangle$, where $p_i,q_i\in \Fun(H)$.
  Since $X'$ is assumed reduced, 
  and we know its support is $(\Pi\times\{0\})\cup (\Lambda\times L)$, 
  $$ I_{X'} = (I_{\Pi \times L} + \langle \ell\rangle) \cap I_{\Lambda\times L} 
  = \langle (p_i), (\ell q_i) \rangle. $$
  Lifting these to generators of $I_X$, we learn
  $$ I_X = \langle (p_i), (\ell q_i + q'_i) \rangle, 
  \qquad \text{for some } q'_i\in \Fun(H). $$
  Projectively completing $L$ to $\Proj \FF[\ell,m]$, and
  closing up $X$ to $\barX \subseteq H \times (L\cup\{\infty\})$, we get
  $$ I_{\barX} = \langle (p_i), (\ell q_i + m q'_i) \rangle. $$

  To determine a fiber of $\barX\to H$, we specialize $H$'s coordinates to 
  values. If all $q_i,q'_i \mapsto 0$ under this specialization,
  then $[\ell,m]$ is free, and the fiber is $\Pone$. 
  Otherwise some equation $\ell q_i + m q'_i=0$ 
  uniquely determines $[\ell,m] \in \Pone$, 
  making the fiber either empty or a point.
  This proves the fibers are connected.

  \emph{(3)$\implies$ the projection is degree $1$ or $0$ to any component.}
  Recall that $\Pi$, by definition, is closed in $H$.
  The map $\barX \to \Pi$ is proper and its image contains an open set, 
  so it is surjective.
  (Note that $X \to \Pi$ itself is usually not surjective;
  part of our task is to describe the points missing from the image.)
  In particular it hits the generic point of each component of $\Pi$, 
  and being reduced has reduced generic fiber, 
  so the connectedness of the fibers gives the claimed degree $1$
  over each component (or $0$ if all the fibers are $\Pone$s).

  So some fibers are points and some are $\Pone$s.
  The latter type are the ones lying over $\Lambda'$.
  But even the point fibers come in two types: those in $X$, and
  those in $(\Lambda \setminus \Lambda')\times \{\infty\}$. This shows that $X$
  misses $(\Lambda \setminus \Lambda')\times L$, and the same is
  true when we project out $L$.

  We wish to show that the map 
  $\pi : X\setminus (\Lambda'\times L) \to \Pi \setminus \Lambda$
  is an isomorphism, not merely bijective. 
  First we consider the case $\Lambda = \nulset$. 
  Since $I_\Lambda = \langle 1 \rangle$, 
  we have $I_X = \langle (p_i), (\ell + q'_i) \rangle$, where the $(p_i)$
  cut out $\Pi$. If there is more than one relation $\ell + q'_i$ on $X$
  then the differences $q'_i - q'_j$ are also satisfied on $\Pi$,
  hence generated by the $(p_i)$, so we may assume there is only
  one such relation. Then we can use it to eliminate $\ell$ and
  determine that $\pi: X \iso \Pi$ is an isomorphism, 
  in this restricted case $\Lambda = \nulset$.
  In the general case, we already know that
  $\pi : X\setminus (\Lambda'\times L) \to \Pi \setminus \Lambda$ 
  is bijective, so 
  we need to check over open sets $U = \Spec A$ covering $\Pi\setminus\Lambda$.
  Replacing $\Pi$ by $U$ (reimbedded as a closed subset of some new $H$)
  and $X$ by $\pi^{-1}(U)$, we can reduce to
  the already solved subcase that $\Lambda = \nulset$.

  This gives an equation in the Grothendieck ring of algebraic varieties,
  $$ [X] = [X \setminus (\Lambda \times L)] + [\Lambda' \times L]
  = [\Pi \setminus \Lambda] + [\Lambda'\times L] 
  = [\Pi] - [\Lambda] + [\Lambda'][L] $$
  and hence
  \begin{align*}
    [X'] &= [\Pi \times \{0\}] + [\Lambda \times L] - [\Lambda \times \{0\}]
    &\text{since $X' 
      = (\Pi\times \{0\})\cup_{\Lambda\times \{0\}} (\Lambda \times L)$} \\
    &= [\Pi] + [\Lambda][L] - [\Lambda] &\\
    &= ([X] + [\Lambda] - [\Lambda'][L]) + [\Lambda][L] - [\Lambda] 
    &\text{by the above}\\
    &= [X] - [\Lambda'][L] + [\Lambda][L] &\\
    &= [X] + [\AA^1]\ [\Lambda\setminus \Lambda']. &
  \end{align*}

  The map $\iota: X \to X'$ defined by
  $$ \iota(h,\ell) := 
  \begin{cases}
    (h,\ell) &\text{if $h\in \Lambda'$} \\
    (h,0) &\text{if $h\notin \Lambda'$} 
  \end{cases}
  $$
  is injective, and its
  image is the complement of $(\Lambda\setminus \Lambda')\times L$.

  \junk{
    So break up 
    $$ \Pi = (\text{image of } X \setminus \Lambda') 
    \coprod (\Lambda \setminus \Lambda') \coprod \Lambda'. $$
    Lying above each $\FF_p$ point in $\Pi$, we count $\FF_p$ points 
    in $X$ and $X'$:
    \begin{align*}
      &\text{image of } X \setminus \Lambda' &
      \Lambda \setminus \Lambda' & & \Lambda' \\
      \text{in }X:&  \quad  1    &     0\quad  & & p\\
      \text{in }X':& \quad  1    &     p\quad  & &  p    
    \end{align*}
  } 

  For the statement on $K$-classes, first define 
  $\barX' := \lim_{t\to 0} t\cdot \barX$ by the same limiting procedure 
  as $X'$. Then note that the action of $G$ on $H\times L$ 
  commutes with the $\Gm$-action, and that
  $K_G$-classes are constant in locally free equivariant families
  such as the one defining $\barX'$.
  (For the nonequivariant statement one can take $G=1$.)
  Hence $[\calO_\barX] = [\calO_{\barX'}]$ as elements of
  $K_G(H\times (L\cup \{\infty\}))$, and in turn 
  $$ [\calO_{\barX'}] = [\calO_{\Pi \times \{0\}}] 
  + [\calO_{\Lambda \times (L\cup\{\infty\})}]
  - [\calO_{\Lambda \times \{0\}}] \quad 
  \in K_G(H\times (L\cup \{\infty\})). $$
  Since the projection $L\cup\{\infty\} \to pt$ takes $\calO_{L\cup\{\infty\}}$
  to $\calO_{pt}$ with no higher direct images, 
  $$ \pi_! [\calO_{\barX'}] = \pi_! [\calO_{\Pi \times \{0\}}] 
  + \pi_! [\calO_{\Lambda \times (L\cup\{\infty\})}]
  - \pi_! [\calO_{\Lambda \times \{0\}}] 
  =  [\calO_{\Pi}] 
  +  [\calO_{\Lambda}] 
  -  [\calO_{\Lambda}]  
  =  [\calO_{\Pi}] 
  \quad \in K_G(\Pi). $$
\end{proof}

In fact the proof (2)$\implies$(3), and the statement about $K$-classes, 
did not use $\Lambda$ reduced; it is enough to assume $X'$ has no
embedded components along $\Pi\times \{0\}$.  This covers the example
of $h^2\ell = 0$ given after the statement of theorem \ref{thm:countdegen}.

We give some examples in which $\Lambda'$ appears.
Let $H = \{(x,y,0)\}$ and $L = \{(0,0,\ell)\}$, 
and let $X = \{(x,y,\ell) : x = \ell y\}$, with $p^2$ points.
Its closure $\barX = \{(x,y,[\ell,m]) : xm = \ell y\}$ is the blowup of $H$
at the origin (so $\Pi = H$). Hence 
$$ \barX \setminus X = \{(x,y,[\ell,0]) : 0 = \ell y\} 
= \{(x,y,[1,0]) : 0 = y\} 
= \{(x,0,[1,0])\} $$ 
so $\Lambda = \{(x,0)\}$. 
The only point in $\Lambda$ hit by the projection $X \to H$ 
is $(0,0)\in \Lambda$, so $\Lambda' = \{(0,0)\}.$
Thus
$$ X' = (H\times \{0\}) \Union_{\Lambda\times\{0\}} (\Lambda\times L) $$
has $p^2 + p^2 - p$ points, and $\Lambda \setminus \Lambda'$ has $p-1$,
giving us the expected
$$ |X'| = p^2 + p^2 - p 
\quad=\quad p^2 + p(p-1) = |X| + p\,\big|\Lambda \setminus \Lambda'\big|. $$

If we take $f = (\ell y-x)(x-1) \in \FF[x,y,\ell]$, then
$f = \ell y(x-1) + x(1-x)$ and
$\Tr(f^{p-1}\bullet)$ defines a splitting of $H\times L$ that 
compatibly splits $X$. But $\Lambda' = \{(0,0)\}$ is not compatibly split
by $\Tr((y(x-1))^{p-1}\bullet)$; the only point that is compatibly split
is $(1,0)$. 

In the notation of the proof, the set $\Lambda'$ is easily seen to be
cut out by the ideal $\langle (p_i), (q_i), (q'_i) \rangle$.
In the following example, this ideal is not even radical, much
less compatibly split. Let $H,L$ be as above, 
with $f = x(\ell y - x) = \ell xy - x^2$, and $X$ the hypersurface $f=0$.
Then $\Pi = H$, 
and $\langle (p_i), (q_i), (q'_i) \rangle 
= \langle (q_i), (q'_i) \rangle = \langle xy, -x^2 \rangle$,
supported on $\Lambda' = \{(0,y)\}$ (which {\em is} compatibly
split by $\Tr(f^{p-1}\bullet)$).

\section{Proof of lemmas \ref{lem:homog} and
  \ref{lem:coordspaces}, and 
  theorems \ref{thm:initprod}, \ref{thm:Grobner}, 
  and \ref{thm:schemeversion}}

\begin{proof}[Proof of lemma \ref{lem:homog}]
  Since the set of compatibly split ideals is finite \cite{Schwede,KumarMehta} 
  hence discrete, if a connected group $G$ preserves the decomposition 
  $R = R^p \oplus \ker \varphi$ on a ring $R$, it must preserve each
  compatibly split ideal. Here we take $G = {\mathbb G}_m$ acting by
  $z \cdot x_i = z^{\lambda_i} x_i$, for which $f$ is assumed to be
  a weight vector, and hence $G$ preserves $\ker \Tr(f^{p-1}\bullet)$.
  Together, we learn this $G$ preserves each compatibly split subscheme $Y$,
  which is equivalent to the statement $Y = \init Y$.
\end{proof}

\begin{proof}[Proof of lemma \ref{lem:coordspaces}]
  Trivially this $\varphi(1)=1$, so $\varphi$ is a splitting.
  For $(\prod_i x_i) r \in \langle \prod_i x_i \rangle$, 
  $$ \varphi\left( (\prod_i x_i) r\right) 
  = \Tr\left( (\prod_i x_i)^{p-1} (\prod_i x_i) r\right) 
  = (\prod_i x_i) \Tr(r) \in \langle \prod_i x_i \rangle, $$
  so $\langle \prod_i x_i \rangle$ is compatibly split.
  The components (also compatibly split) of that ideal define the
  coordinate hyperplanes, whose intersections (also compatibly split)
  are the coordinate subspaces, whose unions (also compatibly split)
  are defined by squarefree monomial ideals.
  
  For the converse, note that $\prod_i x_i = \init \prod_i x_i$ for
  any weighting $\lambda$, hence by lemma \ref{lem:homog}
  a compatibly split subscheme $Y$ must have $Y = \init Y$ 
  for any weighting $\lambda$, which forces $Y$ to be a coordinate subspace.
  (The same argument applies to the ``standard splitting'' of
  any toric variety.)
\end{proof}

For $X$ a Frobenius split scheme, let $\YY_X$ be its set of 
compatibly split subvarieties.
Then there is an
associated decomposition $X = \coprod_{Y\in \YY_X} Y^\circ$,
where $Y^\circ = Y \setminus \union_{Z \in \YY_X, Z\subsetneq Y} Z$.
We point out that
this is a \defn{stratification}, meaning that every closed stratum $Y$ 
is the union $\coprod_{Z^\circ \subseteq Y} Z^\circ$ of the open strata
contained fully in it. Proof: $Y$ is certainly the union
$\coprod_{Z^\circ} (Y \cap Z^\circ)$ of its intersection with the
open strata meeting it. But the components of $Y\cap Z$ are again
elements of $\YY_X$ by the properties of Frobenius splitting, 
so it suffices to use $Z\subseteq Y$. QED.
% Call this the \defn{stratification by compatibly split subvarieties}.

Within the geometric vertex decomposition context, 
we can define a sort of inverse to the map $\pi_{f,\init}$ from 
theorem \ref{thm:frobdegen} part (3), in the following proposition.

\begin{Proposition}\label{prop:algdegen}
  Let $H = \AA^{n-1}_{\FF_p}$ and $L = \AA^1_{\FF_p}$,
  with coordinates $h_1,\ldots,h_{n-1}$ and $\ell$.
  Let $f \in \Fun(H\times L) = \FF_p[h_1,\ldots,h_{n-1},\ell]$ 
  be of degree $n$, where $f = \ell g_1 + g_2$ 
  and $g_1,g_2 \in \Fun(H) = \FF_p[h_1,\ldots,h_{n-1}]$.

  Assume that $\Tr_H(g_1^{p-1}\bullet)$ defines a Frobenius splitting on
  $\Fun(H)$.
  Then $\Tr_{H\times L}(f^{p-1}\bullet)$ defines a Frobenius splitting on
  $\Fun(H\times L)$.
  Let $\YY_{H,g_1}, \YY_{\AA^n,f}$ denote the corresponding sets 
  of compatibly split subschemes. Let $\pi:H\times L\to H$ denote
  the linear projection. Then the map
  $$ Y \in \YY_{\AA^n,f}
  \qquad\mapsto\qquad \overline{\pi(Y)} $$
  takes values in $\YY_{H,g_1}$, and is injective on 
  $\{Y \in \YY_{\AA^n,f} : Y$ is not of the form $P\times L, P\subseteq H\}$.
  Let $\YY^{\not\supseteq \bullet\times L}_{\AA^n,f}$ denote this subset
  of $\YY_{\AA^n,f}$.
\end{Proposition}

\begin{proof}
  Let $\lambda_\ell$ assign weight $1$ to $\ell$ and weight $0$ to any $h_i$,
  and let $\init_\ell g$, $\init_\ell J$ denote the corresponding initial 
  term or ideal, e.g. $\init_\ell f = \ell g_1$. 
  Then by theorem \ref{thm:frobdegen}
  $\Tr_{H\times L}(f^{p-1}) = \Tr_{H\times L}((\ell g_1)^{p-1})$,
  which in turn is $\Tr_H(g_1^{p-1})$. So $\Tr_H(g_1^{p-1}\bullet)$ defines
  a splitting on $H$ iff $\Tr_{H\times L}(f^{p-1}\bullet)$ does on $H\times L$.

  This splitting $\Tr_{H\times L}((\ell g_1)^{p-1}\bullet)$ compatibly splits 
  the hyperplane $H\times\{0\}$, and the induced splitting on $H$
%  $\FF_p[h_1,\ldots,h_{n-1},\ell] / \langle \ell\rangle$
  is $\Tr(g_1^{p-1}\bullet)$, under the identification of
  $\FF_p[h_1,\ldots,h_{n-1},\ell] / \langle \ell\rangle$
  with $\FF_p[h_1,\ldots,h_{n-1}]$.

  Let $Y \in \YY_{\AA^n,f}$, and $Y'$ its limit as studied
  in theorem \ref{thm:countdegen}. That is, if their ideals are
  $I_Y,I_{Y'}$, then $I_{Y'} = \init_\ell I_Y$.
  So again by theorem \ref{thm:frobdegen},
  $Y'$ is compatibly split with respect to the splitting
  $\Tr((\init_\ell f)^{p-1}\bullet) = \Tr((\ell g_1)^{p-1}\bullet)$.
  Since $Y'$ is therefore reduced, 
  $Y' = (\Pi\times \{0\}) \union (\Lambda\times L)$ where $\Pi,\Lambda$
  are also as in theorem \ref{thm:countdegen}.
  Since $H \times \{0\}$ is also compatibly split by 
  $\Tr((\ell g_1)^{p-1}\bullet)$, 
  we see that $(H \times \{0\})\cap Y' = \Pi\times \{0\}$ 
  is compatibly split by $\Tr((\ell g_1)^{p-1}\bullet)$.
  Since $Y$ is irreducible (being an element of $\YY_{\AA^n,f}$),
  so is $\Pi = \overline{\pi(Y)}$. All together this shows
  that $\overline{\pi(Y)} \in \YY_{H,g_1}$ as claimed.

  To see the claimed injectivity, let $Y_1,Y_2 \in \YY_{\AA^n,f}$
  with $\overline{\pi(Y_1)} = \overline{\pi(Y_2)} =: \Pi$, and 
  neither $Y_1$ nor $Y_2$ contain components of the form $P\times L$. 
  Let $Y = Y_1 \cup Y_2$,
  again compatibly split.
  Since their corresponding degenerations $Y'_1,Y'_2,Y'$ 
  are compatibly split (with respect to $\Tr((\ell g_1)^{p-1}\bullet)$)
  hence reduced, the projections $Y_1,Y_2,Y \to \Pi$ are each degree $1$
  over every component of $\Pi$ by theorem \ref{thm:countdegen}.
  If $Y_1\neq Y_2$ over some component of $\Pi$, 
  the degree of $Y_1\cup Y_2 \to \Pi$ would be the
  sum of the two degrees, so each $2$ not $1$. Since $Y_1,Y_2$ agree over
  each component of their common projection, and contain
  no components of the form $P\times L$, we obtain $Y_1=Y_2$.
\junk{
  We may assume that the decomposition $H = \coprod H_\gamma$ {\em is}
  the stratification by compatibly split subvarieties, rather than
  one that refines it; if one refines some $H_\gamma$ to a union
  $\coprod_\delta (H_\gamma)_\delta$ then one can similarly refine each 
  $A_{(\gamma,i)} 
  = \coprod_\delta A_{(\gamma,i)} \cap ((H_\gamma)_\delta \times L)$.
  
  As indicated, the decomposition
  $\AA^n = \coprod_{(\gamma,i)} A_{(\gamma,i)}$ is to be a refinement
  of the decomposition $\coprod_{\gamma} H_{\gamma}\times L$.
  Already any $Y \in \YY_{\AA^n,f}$ of the form $\Pi\times L$
  is a union of these strata, so in defining $\{A_{(\gamma,i)}\}$ we need to 
  keep in mind the remaining $\{Y\}$, 
  in $\YY^{\not\supseteq \bullet\times L}_{\AA^n,f}$.

  There are two cases. If $\overline{H_\gamma} = \pi(Y)$ for
  some $Y \in \YY^{\not\supseteq \bullet\times L}_{\AA^n,f}$, then let
  $A_{(\gamma,0)} = Y \cap (H_\gamma \times L)$. 
  Otherwise let $A_{(\gamma,0)} = H_\gamma \times \{0\}$.
  Either way, 
  let $A_{(\gamma,1)} = (H_\gamma \times L) \setminus A_{(\gamma,0)}$.

  Now we claim that each $Y \in \YY^{\not\supseteq \bullet\times L}_{\AA^n,f}$
  is the union of the $A_{(\gamma,i)}$ that it intersects.
  {why would that be? I think that requires $\Lambda'$ to
    be compatibly $\init f$-split.}
}
\end{proof}

\begin{proof}[Proof of theorem \ref{thm:initprod}]
  If $\deg f < n$, we can multiply it by the product of the variables
  not appearing in $\init f$, thereby only increasing the set of $\{Y\}$
  obtained by the algorithm. We thereby reduce to the case $\deg f = n$
  and $\init f = \prod_{i=1}^n x_i$.

  If $p\notdivides c$, then $c\prod_i x_i$ is still a term in
  the reduction of $f$ mod $p$, and still the initial term.
  Hence $\init(f)^{p-1} \equiv c^{p-1}\prod_i x_i^{p-1} 
  \equiv \prod_i x_i^{p-1} \bmod p$ by Fermat's Little Theorem.

  Since $\Tr(\prod_i x_i^{p-1}\bullet)$ defines a splitting
  (indeed, the standard splitting), so does $f \bmod p$,
  by theorem \ref{thm:frobdegen} part (1). 
  Part (2) says that for any compatibly split $I \leq \FF_p[x_1,\ldots,x_n]$,
  the initial ideal $\init(I)$ is compatibly split 
  by $\Tr(\init(f)^{p-1}\bullet) = \Tr(\prod_i x_i^{p-1}\bullet)$,
  hence is a Stanley-Reisner ideal by lemma \ref{lem:coordspaces}.

  Now let $I \leq \integers[x_1,\ldots,x_n]$ be one of the ideals
  constructed using the algorithm in corollary \ref{cor:allradical}.
  Only finitely many quotient rings $R/I'$ are encountered on the way to $R/I$,
  each of which is flat over an open set in $\Spec \integers$.
  Hence if we restrict to primes $p$ in this finite intersection of open sets,
  when we work the algorithm mod $p$ we encounter $I \bmod p$.
  Let $S$ be the set of primes we are avoiding so far.

  Let $(g_i)$ be a Gr\"obner basis for 
  $S^{-1}I \leq (S^{-1}\integers)[x_1,\ldots,x_n]$.
  If we increase $S$ to include the primes dividing the coefficients
  on the initial terms $(\init g_i)$, then we can rescale the $(g_i)$ to
  make their initial coefficients $1$, and insist that no 
  $\init g_i$ divides any $\init g_j$, $j\neq i$. 
  For any $p\notin S$, these properties hold also for $(g_i\bmod p)$.

  As observed above, $\init(I\bmod p)$ is a Stanley-Reisner ideal, $p\notin S$.
  Hence the initial monomials $\init (g_i \bmod p)$ are squarefree.
  So the initial monomials $\init g_i$ are themselves squarefree.
  This proves that away from the bad primes in $S$, 
  $\init I$ is a Stanley-Reisner ideal. In particular $\init I$, and $I$
  itself, are radical over $S^{-1} \integers$, as was to be shown.
\junk{
  The algorithm in corollary \ref{cor:allradical} only constructs
  finitely many subschemes $Y$, so all together there are only finitely many
  primes to avoid, which we do hereafter.
  ...
  Finally, the bijection $\AA^n \to \AA^n$ taking each 
  $Y$ into $\init Y$ (not just $\init_\ell Y$) is made by composing
  $\iota$ with ${\bf 1} \times \iota_H$, where $\iota_H$
  is the bijection of $H$ to itself already constructed by induction.
}

  Now take $\lambda$ to be the lexicographic weighting.
  In this case, $\init I = \init_n \init_{n-1} \cdots \init_1 I$,
  where $\init_j$ is defined using the weighting 
  $\lambda_j = (0,\ldots,0,1,0,\ldots,0)$. 

  For each $j$ in turn, 
  let $L$ be the $j$th coordinate line $\{(0,\ldots,0,*,0,\ldots,0)\}$, 
  and $H$ the complementary coordinate hyperplane.
  Define $\ell_f : H \to L$ by 
  $$ \ell_f(h) :=
  \begin{cases}
    \ell &\text{if a minimal $Y \subseteq \AA^n$ meeting $\{h\}\times L$
      meets it in the single point $(h,\ell)$} \\
    0 &\text{otherwise}
  \end{cases}
  $$
  where $Y$ ranges over the subschemes created by the algorithm.
  As there are only finitely many $Y$, this map is constructible.
  Then define the bijection $\iota_j : \AA^n \to \AA^n$ by 
  $$ \iota_j(h + \ell) := h + \ell -\ell_f(h). $$ 

  We claim that $\ell_f$ (and thus $\iota_j$) is well-defined. 
  First we need to be sure that it is defined everywhere. 
  For every $h$, {\em some} $Y$ meets and even contains $\{h\}\times L$, 
  namely $Y=\AA^n$. 
  If no $Y$ meets $\{h\}\times L$ in a finite scheme, then $\ell_f(h)=0$. 
  Otherwise let $Y_h$ be the union of the $Y$ that meet $\{h\}\times L$ 
  in a finite scheme. (Indeed, it is the unique largest such $Y$.)
  Then since $\init_j Y_h$ is reduced, by the (2)$\implies$(3) part
  of theorem \ref{thm:countdegen}, $Y_h$ intersects $\{h\}\times L$
  in a point, $(h,\ell_h)$. This also shows that the choice of
  minimal $Y$ does not matter, since any such $Y$ lies in $Y_h$.

\junk{
  We note for later use that $\{h : \ell_f(h)=0\}$ is closed in $H$.
  {is this true?}
}

  Now we claim that $\iota_j$ takes $Y$ into $\init_j Y$.
  This is because this $\iota$ agrees with the $\iota$ from
  theorem \ref{thm:countdegen} except possibly on 
  $\Lambda'\times L \subseteq Y$, 
  where it is merely shifted by an element of $L$.

  When we then define $\iota := \iota_n \circ \cdots \circ \iota_1$, 
  we learn inductively that $\iota(Y) \subseteq Y'$.
\junk{
  Now we use $\init_\ell f$ and induction to define a decomposition
  $H = \coprod_{\gamma\subseteq \{2,\ldots,n\}} H_\gamma$ of $H$
  that restricts to a decomposition of each of its $Y$. 
  We decompose each $H_\gamma \times L$ 
  into the graph of $\ell_f: H_\gamma \to L$ and its complement.

  To see that these are tori 
  (isomorphic to $H_\gamma$ and $H_\gamma \times \Gm$)
  we need to show that on each $H_\gamma$, 
  the map $\ell_f$ is regular (not just constructible).
  By induction we may assume that $H_\gamma$ is a torus,
  and in particular irreducible.

  There is an easy case, where each $Y$ that meets $H_\gamma \times L$
  contains it. Then $\ell_f \equiv 0$ on $H_\gamma$.

  Otherwise let $Y$ m
}
\end{proof}

It is quite unfortunate that the $\{\Lambda'\}$ from 
theorem \ref{thm:countdegen} are not always compatibly split in $H$, 
which is to say, one cannot stratify the varieties in
$\YY^{\not\supseteq \bullet\times L}_{\AA^n,f}$ using the varieties in
$\YY^{irr}_{H,g_1}$ (or more precisely, their 
$Y^\circ = Y \setminus \union_{Z\subsetneq Y} Z$). 
If this were true, one could use induction to give
$\AA^n$ a paving by tori, compatibly paving each $Y$, pulled back
using $\iota$ from the standard paving (and with $\iota$ regular on
each torus stratum).  The well-known example \cite{Deodhar} of such a
simultaneous paving suggests there should be a criterion guaranteeing that
the $\{\Lambda'\}$ are split, that would apply to the Kazhdan-Lusztig
case considered in section \ref{ssec:kl} and \cite{Deodhar}.
However, the subtle example at the end of section \ref{sec:pfcountdegen}
is of this type, making such a criterion difficult to imagine.

\begin{proof}[Proof of theorem \ref{thm:Grobner}]
  \begin{enumerate}
  \item By theorem \ref{thm:initprod}, $\init Y$ is Stanley-Reisner
    and reduced, hence its ideal can be generated by squarefree monomials.
    Lifting those to generators of $Y$'s ideal, we get the desired
    Gr\"obner basis.
  \item Set-theoretically, it is plain that the intersection of
    $\{Z \in \YY_f : Z \geq Y, Z$ basic in $\YY_f\}$ 
    is $Y$. But since this intersection is compatibly split, 
    it is reduced, so the equation holds scheme-theoretically.
  \item Since any intersection of compatibly split ideals is compatibly
    split, and by theorem \ref{thm:initprod} their initial ideals
    are radical, corollary \ref{cor:Grobner} from section \ref{sec:Trinit} 
    allows us to concatenate their Gr\"obner bases.
  \end{enumerate}
\end{proof}

We mention a property of $\pi_{f,\init}: \YY_{\init f} \onto \YY_f$:
for all $Y' \in \YY_{\init f}$, $\dim \pi_{f,\init}(Y') \geq \dim Y'$.
Proof: $Y' \subseteq \init \pi_{f,\init}(Y')$, and $\init$ preserves
dimension. The examples in figure \ref{fig:elliptic}
show that the inequality may be strict.

\begin{proof}[Proof of theorem \ref{thm:schemeversion}]
  As this is the only place in the paper that we consider Frobenius
  splittings on schemes rather than on affine space, we will not 
  build up all the relevant definitions, but assume the reader is
  familiar with \cite[section 1.3]{BK} and \cite{KumarMehta}.
  \begin{enumerate}
  \item The proof is exactly the same as in \cite[proposition 1.3.11]{BK},
    once one includes \cite[remark 1.3.12]{BK} to relax 
    nonsingularity to normality. Completeness enters as follows:
    the scheme-theoretic analogue of $\Tr(f^{p-1})$ is a global
    function on $X$, and the condition on $f$ ensures that
    $\Tr(f^{p-1})$ is locally of the form $1 +$ higher order. 
    But since $X$ is complete, normal, and irreducible (being a variety), 
    any global function on $X_{reg}$
    is constant, so the higher order terms vanish and $\Tr(\sigma^{p-1}) = 1$.
  \item Write $f = \prod_i t_i\ (1 + \sum_{e} c_{e} {\bf t}^{e})$,
    where the $\{e\}$ are exponent vectors, all entries $\geq -1$. 
    The condition on $f$ says that $\lambda \cdot e < 0$ for 
    each summand. Write $f^{p-1} 
    = \prod_i t_i^{p-1}\ (1 + \sum_{s} d_{s} {\bf t}^{s})$.
    The exponent vectors $s$ %on the terms in $(f/\prod_i t_i)^{p-1}$ 
    are sums of $p-1$ of the vectors $\{e\}$.

    What can such an $s$ look like? It is integral, with all 
    entries $\geq 1-p$. Since $\lambda \cdot s < 0$, and each
    entry of $\lambda$ is $\geq 0$, there must be some entry of $s$
    in $(-p,0)$. Hence $s$ cannot be $p$ times another vector.

    Consequently the only term of 
    $\prod_i t_i^{p-1}\ (1 + \sum_{s} d_{s} {\bf t}^{s})$
    that contributes in $\Tr(f^{p-1})$ is the first one, 
    so $\Tr(f^{p-1})=1$ locally. By irreducibility, $\Tr(f^{p-1})=1$.
  \end{enumerate}
  In either case, the same argument from theorem \ref{thm:Trpointcount}
  guarantees that the splitting defined by $\sigma^{p-1}$ does
  compatibly split the divisor $\sigma=0$. Normality lets us extend
  the splitting from $X_{reg}$ to $X$ \cite[Lemma 1.1.7(iii)]{BK}.

  To see the uniqueness of the splitting, 
  the proof of \cite[proposition 2.1]{KumarMehta}
  shows that a section $\gamma$ of $\omega^{1-p}$ defining a Frobenius splitting
  on a nonsingular variety splits a divisor $\{\sigma = 0\}$ iff 
  $\gamma$ is a multiple of $\sigma^{p-1}$. By the assumption that
  $\{\sigma=0\}$ is anticanonical, this multiple must be by a global function.
  Completeness and irreducibility ensures this global function is a constant.  
\end{proof}

\section{Examples, theorem \ref{thm:kl},
  and proposition \ref{prop:multfreesplit}}\label{sec:examples}

In sections \ref{ssec:smalldim}-\ref{ssec:kl}
we investigate the condition $\init(f) = \prod_i x_i$ in examples,
sometimes using Macaulay 2 \cite{M2}. 
The most important family of examples is the Kazhdan-Lusztig varieties
in section \ref{ssec:kl}.
In section \ref{ssec:multfree} we mention a corollary
about Brion's ``multiplicity-free'' subvarieties of a flag manifold
\cite{BrionMultFree}.

\subsection{Small dimensions}\label{ssec:smalldim}

We begin with a general weighting $\lambda_1,\ldots,\lambda_n$ of the
variables, so $\prod x_i^{e_i}$ has weight $\sum \lambda_i e_i$. 
Write $\prod x_i^{e_i} \rhd \prod x_i^{f_i}$ if its
weight is greater.

Without loss of generality we may assume the $(\lambda_i)$ 
to be in decreasing order. 
Then for any two monomials $\prod x_i^{e_i}$ and $\prod x_i^{f_i}$,
if $\sum_{i\leq j} e_i \geq \sum_{i\leq j} f_i$ for all $j=1,\ldots,n$, 
we may be sure that $\prod x_i^{e_i} \rhd \prod x_i^{f_i}$.

Once we choose a specific $\lambda$ (generic enough that 
$\prod_i x_i$ does not have the same weight as any other monomial),
the condition that $\lt(f) = \prod_i x_i$ forces us to put 
coefficient $0$ on any $m$ with $m \rhd \prod_i x_i$. 
Different choices of $\lambda$ will lead to different sets of
allowed $m$, but as our interest is not in $\{\lambda\}$ but
in the set of varieties to which theorem \ref{thm:initprod}
applies, it is enough for us to consider the maximal allowed subsets
of monomials.

\subsubsection{n=2.} In this case the set of permitted monomials
is already uniquely specified: $\{x_1 x_2, x_2^2\}$. 
So the possible polynomials are $f = x_1 x_2 + c_{20} x_2^2$,
defining the two (distinct) points $\{[1,0], [-c_{20},1]\}$ in $\Pone$.
%Note that these points are distinct because $c_{11} x_1x_2$ is
%supposed to be the leading term, i.e. $c_{11} \neq 0$.

\subsubsection{n=3.} The order considerations so far already tell us 
which monomials $m$ have $m \rhd x_1 x_2 x_3$ and which have
$x_1 x_2 x_3 \rhd m$, except for $m = x_2^3$. 
One can find decreasing $\lambda$ for either condition
$x_2^3 \rhd x_1 x_2 x_3$ or $x_1 x_2 x_3 \rhd x_2^3$,
but since we only want to
consider the maximal cases, we ask that $x_1 x_2 x_3 \rhd x_2^3$,
which is achieved by e.g. $(\lambda_1,\lambda_2,\lambda_3) = (3,1,0)$.

Now we are considering plane curves of the form
%$$ F = a x^2 y + b x y z + c x^3 
%$$ f = c_{210} x^2 y + c_{111} x y z + c_{300} x^3 
%+ d x^2 z + e x z^2 + f z^3 $$
%+ c_{201} x^2 z + c_{102} x z^2 + c_{003} z^3 $$
$$ f = x_1 x_2 x_3 + c_{030} x_2^3+c_{021} x_2^2 x_3+c_{102} x_1 x_3^2
+c_{012} x_2 x_3^2+c_{003} x_3^3. $$
%with $c_{111}\neq 0$. 
Being cubics and Frobenius split, they are elliptic
curves with at worst nodal singularities \cite[1.2.4--1.2.5]{BK}.
Each curve has $[1,0,0]$ as a singular point, so they are indeed nodal.

If we write $f$ as
$ f = x_1 (x_2 x_3 +c_{102} x_3^2)
+ c_{030} x_2^3+c_{021} x_2^2 x_3 +c_{012} x_2 x_3^2+c_{003} x_3^3 $
and degenerate as in proposition \ref{prop:gvdsplit}, 
we see $I_\Lambda = \langle x_3 \rangle \cap \langle x_2 +c_{102} x_3 \rangle$.
For generic values of the coefficients, $\Lambda' = \{[1,0,0]\}$.

Being cubics, these curves can only have two more nodes.
To find them we decompose
the ideal generated by $f$ and its derivatives, using the Macaulay 2 command

\centerline{\tt 
  decompose ((ideal f) + ideal diff(matrix $\{\{x_1, x_2, x_3\}\}$, f))}
\noindent
Then for each component $c$ of the possible singular set, 
we {\tt eliminate($\{x_1,x_2,x_3\}$,c)} to see for what $f$ the
singularities arise. The components turn out to be
$\{c_{003}=0\}$ and $\{c_{300} + c_{210}^2 c_{102}  
= c_{210} c_{201} + c_{210}^3 c_{003}\}$.

If $c_{003}=0$, the cubic breaks
into the $x=0$ line and a conic. 
%$ c_{210} x y + c_{111} y z + c_{300} x^2 
%+ c_{201} x z + c_{102} z^2 = 0 $,
%meeting at $[0,1,0]$ and $[0,-c_{102},c_{111}]$.
%One might worry that as the nodes collide, we will get a non-nodal
%singularity forbidden by Frobenius splitting, but this would require
%$c_{111}\to 0$ which would remove $xyz$ as the leading term.
%
%If  $c_{111}^3 c_{300} + c_{210}^2 c_{111} c_{102}  
%= c_{210} c_{111}^2 c_{201} + c_{210}^3 c_{003}$,
If 
$c_{030}c_{102}^3+c_{012}c_{102}
=c_{021}c_{102}^2+c_{003}$,
%If $b^3 c + a^2 b e  = a b^2 d + a^3 f$, 
the curve breaks into the 
$x_2 + c_{102}x_3 = 0$ line and a conic, and $\Lambda'$ contains the
new node.
%$c_{210} x + c_{111} z = 0$
%meeting at $[0,1,0]$ and $[b(be-3af), - bde-2ae^2 + 9bcf -6adf, -a(be-3af)]$.
%
%If $c_{030}=0$ and
%$c_{030}c_{102}^3+c_{012}c_{102}c_{111}^2
%=c_{021}c_{102}^2c_{111}+c_{003}c_{111}^3$
%%%%%also,
If both are true,
the curve is a cycle of three lines.
%connecting the nodes at $[0,1,0]$, $[0,-e,b]$, 
%and $[b^2, -bd-2ae , -ab]$. 
%Again, because $b\neq 0$ the third node cannot fall into the other two.

%\subsubsection{$n=4$.}

\subsection{Matrix Schubert varieties}\label{ssec:matrixSchubs}

The \defn{matrix Schubert variety} $\barX_\pi \subseteq M_n$
is the closure of $B_- \pi B_+$ inside the space $M_n$ of all matrices,
where $B_+$ (respectively $B_-$) denotes the Borel group of upper 
(respectively lower) triangular matrices, and $\pi$ is a 
permutation matrix.
These varieties were introduced in \cite{Fulton92}, 
where their corresponding radical ideals $I_\pi$ were determined.
They have a couple of relations to flag manifold Schubert varieties; 
in particular $(\barX_\pi \cap GL(n))/B_+$ is 
the usual Schubert variety $X_\pi \subseteq GL(n)/B_+$.
Hence the codimension of $\barX_\pi$ in $M_n$ is
the Coxeter length $\ell(\pi)$.

It is easy to give examples of weightings $\lambda$ on the 
matrix coordinates $(x_{ij})$ such that 
for the determinant $d$ of any submatrix, $\init(d)$ is the product
of the entries on the antidiagonal (times a sign). In \cite{KM} we called these
\defn{antidiagonal term orders}, and showed that each $\init(I_\pi)$
is Stanley-Reisner over $\integers$. In this section we do this part
of \cite{KM} (though only over $\rationals$) much more easily using
theorem \ref{thm:initprod}.

Let $M = (m_{ij})_{i,j=1,\ldots,n}$ be an $n\times n$ matrix
of indeterminates, and let $d_{[i,j]}$ denote the determinant of
the submatrix consisting of rows and columns $i,i+1,\ldots,j$ from $M$.
Let
$$ f := \prod_{i=1}^{n-1} d_{[1,i]} \prod_{j=1}^n d_{[j,n]}. $$
This is homogeneous of degree 
$1+2+\ldots+(n-1)+n+(n-1)+\ldots+2+1 = n^2$. 
Let 
$$ d'_i := (-1)^{i-1} \prod_{j,k:\, j+k=i+1} m_{jk} $$
be the product of the $i$th antidiagonal of $M$ (with a not particularly
important sign), so
$$ \init(d_{[1,i]}) = d'_i, \quad
\init(d_{[j,n]}) = d'_{n+j-1} $$
and since the set of matrix entries is the union of the antidiagonals,
$$
  \init(f) 
  = \prod_{i=1}^{n-1} \init(d_{[1,i]}) \prod_{j=1}^n \init(d_{[j,n]}) 
  = \prod_{i=1}^{n-1} d'_i
  \prod_{j=1}^n d'_{n+j-1}
  = \prod_{i=1}^{2n-1} d'_i = \pm \prod_{i,j} m_{ij}. $$
Hence theorem \ref{thm:initprod} applies.

Next we apply the algorithm from corollary \ref{cor:allradical}
to the ideal $\langle f\rangle$. We will restrict to the components
$\{ \langle d_{[1,i]} \rangle \}$, which define exactly the 
matrix Schubert varieties $\barX_{r_i}$ associated to the simple reflections.
Each of these is $B_-\times B_+$-invariant under the left/right action,
and this invariance persists as we intersect, decompose, and repeat.

Copying \cite[theorem 2.3.1]{BK},
we claim that every matrix Schubert variety $\barX_\pi$ is
produced by this algorithm. The proof is by induction on the length of $\pi$;
we are given the $\ell(\pi)=1$ base case to start with. 
We need the combinatorial fact that
for any $\pi$ with $\ell(\pi)>1$, there exist at least two permutations
$\rho \neq \rho'$ covered by $\pi$ in the Bruhat order
\cite[lemma 10.3]{BGG}.
We know by induction that their matrix Schubert varieties 
have already been already been produced.
Now $\barX_\rho \cap \barX_{\rho'} \supseteq \barX_\pi$, and
$\dim \left(\,\barX_\rho \cap \barX_{\rho'}\right) < \dim \barX_\rho 
= \dim \barX_\pi + 1$, so $\barX_\pi$ must be a component of
$\barX_\rho \cap \barX_{\rho'}$. Therefore it too is produced
by the algorithm.

To apply theorem \ref{thm:Grobner}, we need to compute the basic
elements of $S_n$, shown in \cite{biGrassmannian} to be those $\pi$
such that $\pi,\pi^{-1}$ are each Grassmannian.\footnote{%
  We only need one, easy, direction. If $\pi$ has descents at both $i$
  and $j$, then $\pi r_i, \pi r_j$ both cover $\pi$ (in opposite
  Bruhat order), and it is easy to see that $\pi$ is their unique
  greatest lower bound.  Also, $\pi \mapsto \pi^{-1}$ is an automorphism
  of Bruhat order.  Hence $\pi$ basic implies $\pi,\pi^{-1}$ Grassmannian.}
For those $\pi$, Fulton's theorem \cite{Fulton92} states that
$\barX_\pi$ is defined by the vanishing of all $a\times a$ determinants 
in the upper left $b\times c$ rectangle, for $a,b,c$ determined by $\pi$ 
(the ``essential set'' of $\pi$ has only the box $\{(b,c)\}$).  
These determinants are already known to form a
Gr\"obner basis for any antidiagonal\footnote{%
  The reference \cite{Sturmfels90} uses diagonal term orders, but the
  ideal is symmetric in the rows, so we can reverse them.}  term order
\cite{Sturmfels90}.
Now part (2) of theorem \ref{thm:Grobner} recovers Fulton's presentation
of the ideals defining general $\barX_\pi$ \cite{Fulton92}, 
and part (3) recovers the main result of \cite{KM}, that Fulton's generators
form a Gr\"obner basis.

Note that while we used only $\prod_{i=1}^{n-1} d_{[1,i]}$ to produce
the matrix Schubert varieties, that polynomial wasn't of high enough
degree to give a splitting, and we needed to flesh it out to $f$.  
It is interesting to note that this function (not $f$) was already
enough to construct a Frobenius splitting on $GL_n \times^B \lie{b}$
in \cite[section 3.4]{MvdK},
when pulled back along $GL_n \times^B \lie{b} \onto \lie{gl}_n$.

Finally, we mention that the definition of $f$ generalizes easily to
the case of rectangular matrices, say $k\times n$ with $k\leq n$. 
Let $d_{[i_1,i_2]}^{[j_1,j_2]}$ denote the determinant of the 
submatrix using rows $[i_1,i_2]$ and columns $[j_1,j_2]$ (so the
previous $d_{[i,j]}$ is this $d_{[i,j]}^{[i,j]}$). Then take
$$ f := \prod_{i=1}^{k-1} d_{[1,i]} 
\prod_{i=1}^{n-k} d_{[1,k]}^{[i,i+k-1]}
\prod_{j=1}^k d_{[j,k]}^{[n-j+1,n]} $$
and the antidiagonal terms again exactly cover the matrix.

\newcommand\im{{\rm im\,}} % image
\subsection{Kazhdan-Lusztig varieties}\label{ssec:kl}
This example requires a fair amount of standard Lie theory, in particular a
pinning $(G,B,B_-,T,\Delta_\pm,W)$ of a connected 
simply connected reductive algebraic group $G$.
A typical simple root will be denoted $\alpha \in \Delta_+$,
with corresponding simple reflection $r_\alpha \in W$,
subgroup $(SL_2)_\alpha$,
one-parameter unipotent subgroup $e_\alpha : \Ga \to B\cap (SL_2)_\alpha$,
and minimal parabolic $P_\alpha = (SL_2)_\alpha B$.
That group has the Bruhat decomposition 
$P_\alpha = B \coprod B \tilde r_\alpha B
= B \coprod (\im e_\alpha) \tilde r_\alpha B$,
where $\tilde r_\alpha$ is a lift of $r_\alpha$ to an element of
the corresponding $(SL_2)_\alpha \leq G$, chosen so that
$$ e_\alpha(c) \mapsto  \begin{pmatrix}   1&c\\ 0&1  \end{pmatrix}, \qquad 
\tilde r_\alpha \mapsto  \begin{pmatrix}   0&-1\\ 1&0  \end{pmatrix} $$
under some isomorphism of $(SL_2)_\alpha$ with $SL_2$,
taking $B_\pm \cap (SL_2)_\alpha$ to upper/lower triangular $2\times 2$ 
matrices.

For $w\in W$ a Weyl group element,
let $X_w^\circ := {B_- wB}/B \subseteq G/B$ and $X_w := \overline{X_w^\circ}$ 
be the associated \defn{Bruhat cell} and \defn{Schubert variety}, 
of codimension $\ell(w)$ (the length of $w$ as a Coxeter element).
Define also the \defn{opposite Bruhat cell} $X^v_\circ := {B v B}/B$
and \defn{opposite Schubert variety} $X^v := \overline{X_v^\circ}$,
of dimension $\ell(v)$.\footnote{%
  As our personal interest is most often in the cohomology classes
  associated to Schubert varieties, we prefer to privilege the codimension
  over the dimension, dictating our convention of Schubert/opposite Schubert.
  This will have the drawback later that containment order on Schubert
  varieties, relevant for computing the ``basic'' elements, 
  is opposite Bruhat order.}
It is known \cite[theorem 2.3.1]{BK} that $G/B$ possesses a Frobenius splitting
(meaning, on its structure sheaf of rings) that compatibly splits
all its Schubert varieties and opposite Schubert varieties.

Consequently, each $X^v_\circ$ is Frobenius split, 
compatibly splitting the \defn{Kazhdan-Lusztig varieties}
$X^v_{w\circ} := X_w \cap X^v_\circ$.
In \cite{schubDegen} we showed that each Kazhdan-Lusztig variety has
a flat degeneration to a Stanley-Reisner scheme, using a sequence
of different coordinate systems. 
In what follows we will derive both of these results 
from theorem \ref{thm:initprod}, using a single identification 
of $X^v_\circ$ with $\AA^{\ell(v)}$.

As in \cite{schubDegen}, these coordinates will depend
on a \defn{reduced word $Q$ for $v$,} 
i.e. a list $(\alpha_1,\ldots,\alpha_{\ell(v)})$ of simple roots
such that $r_{\alpha_1} \cdots r_{\alpha_{\ell(v)}} = v$. 
Associated to $Q$ is a \defn{Bott-Samelson manifold} 
$BS^Q := P_{\alpha_1} \times^B \cdots \times^B P_{\alpha_{\ell(v)}}/B$
and birational map $\beta_Q: BS^Q \onto X^v$, taking 
$[p_1,\ldots,p_{\ell(v)}] \mapsto \left( \prod_{i=1}^{\ell(v)} p_i\right) B/B$.
In particular, we can use $\beta_Q$ to define an isomorphism
$$ \AA^{\ell(v)} \to X^v_\circ, \qquad
(c_1,\ldots,c_{\ell(v)}) 
\mapsto \left( \prod_{i=1}^{\ell(v)} (e_{\alpha_i}(c_i) \tilde r_{\alpha_i})
\right) B/B. $$

Define $\tilde\beta_Q: \AA^{\ell(v)} \to G$ by
$\tilde\beta_Q(c_1,\ldots,c_{\ell(v)}) 
= \prod_{i=1}^{\ell(v)} (e_{\alpha_i}(c_i) \tilde r_{\alpha_i})$.
For $\lambda$ a dominant weight of $G$, pick $\vec v_\lambda$ a high
weight vector of the irrep $V_\lambda$ (with highest weight $\lambda$), 
and $\vec v_{-\lambda}$ a low weight vector of $(V_\lambda)^*$. 
We can scale $\vec v_{\pm\lambda}$ to ensure 
$\langle \vec v_{-\lambda}, \vec v_\lambda \rangle = 1$.
Define $m_\lambda : G \to \AA^1$ by 
$$ m_\lambda(g) := 
\langle \vec v_{-\lambda}, g\cdot \vec v_\lambda \rangle. $$
It is easy to see that $m_\lambda$ does not depend on the choices of
$\vec v_{\lambda}, \vec v_{-\lambda}$. If $\lambda$ is a fundamental
weight $\omega$ (as indeed it will be), this is one of the
``generalized minors'' defined by Fomin and Zelevinsky; in particular,
if $\omega_i$ is the $i$th fundamental weight of $SL_n$, then
$m_{\omega_i}$ is the determinant of the upper left $i\times i$ submatrix.

\begin{Lemma}\label{lem:momega}
  Let $\alpha$ be a simple root and $\omega$ the corresponding
  fundamental weight. Let $\lambda$ be a dominant weight,
  so $\langle \alpha,\lambda \rangle \geq 0$.
  \begin{enumerate}
  \item The divisor $m_\omega = 0$ in $G$ is the preimage of
    the Schubert divisor $X_{r_\alpha} \subseteq G/B$.
  \item If $\langle \alpha,\lambda \rangle = 0$, then 
    $m_\lambda(e_{\alpha}(c) \tilde r_\alpha g) = m_\lambda(g)$ for
    all $g\in G, c\in \Ga$.
  \item If $\langle \alpha,\lambda \rangle = 1$, then
    $m_\lambda(e_{\alpha}(c)\, \tilde r_\alpha g) 
    = c\, m_\lambda(g) + m_\lambda(\tilde r_\alpha g)$ 
    for all $g\in G, c\in \Ga$.
  \end{enumerate}
\end{Lemma}

\begin{proof}
  \begin{enumerate}
  \item It is easy to see that the function $m_\lambda$
    is invariant up to scale under the left/right action of $B_- \times B$.
    So the divisor $m_\lambda=0$ is the preimage of some $B_-$-invariant 
    divisor $D$ on $G/B$, necessarily some linear combination of
    the Schubert divisors. The coefficient of $X_{r_\alpha}$ in $D$
    can be determined by restricting the class of $D$ to the opposite
    Schubert curve $X^{r_\alpha}$, and turns out to be 
    $\langle \alpha, \lambda \rangle$. In particular, we get
    $D = X_{r_\alpha}$ exactly if $\lambda = \omega$.
  \item 
    $$ m_\lambda(e_{\alpha}(c) \tilde r_\alpha g) 
    = \langle \vec v_{-\lambda}, 
    e_{\alpha}(c) \tilde r_\alpha g \cdot \vec v_\lambda \rangle
    = \langle (e_{\alpha}(c) \tilde r_\alpha)^{-1}\cdot \vec v_{-\lambda}, 
    g\cdot \vec v_\lambda \rangle $$
    The condition on $\lambda$ says that $\vec v_{-\lambda}$ is a weight vector
    not only for $B_-$ but for the opposite minimal parabolic $P_{-\alpha}$.
    Then use the fact that $e_{\alpha}(c), \tilde r_\alpha$ are
    elements of the commutator subgroup of $P_{-\alpha}$ (indeed, of 
    $(SL_2)_\alpha$) to see that 
    $(e_{\alpha}(c)\tilde r_\alpha)^{-1}\cdot \vec v_{-\lambda} 
    = \vec v_{-\lambda}$,
    hence
    $$ \langle (e_{\alpha}(c) \tilde r_\alpha)^{-1} \cdot \vec v_{-\lambda}, 
    g\cdot \vec v_\lambda \rangle
    = \langle \vec v_{-\lambda}, g\cdot \vec v_\lambda \rangle = m_\lambda(g).$$
  \item The condition on $\lambda$ tells us that 
    $r_\alpha \cdot (-\lambda) = (-\lambda) + \alpha$, i.e.
    the $\alpha$-string through $-\lambda$ in $(V_\lambda)^*$
    is $\{-\lambda, -\lambda + \alpha\}$.
    Hence the representation of $(SL_2)_\alpha$ on the sum of
    these two extremal (hence $1$-dimensional) 
    weight spaces is isomorphic to the defining representation,
    in which
    $$ e_{\alpha}(-c) \cdot \vec v_{-\lambda} 
    = c \tilde r_\alpha \vec v_{-\lambda} + \vec v_{-\lambda} . $$
    Hence 
    \begin{align*}
      m_\lambda(e_{\alpha}(c) \tilde r_\alpha g) 
      &= \langle \vec v_{-\lambda}, 
      e_{\alpha}(c) \tilde r_\alpha g \cdot \vec v_{-\lambda} \rangle \\
      &= \langle e_{\alpha}(-c)  \vec v_{-\lambda}, 
      \tilde r_\alpha g \cdot \vec v_{-\lambda} \rangle \\
      &= \langle c \tilde r_\alpha \vec v_{-\lambda} + \vec v_{-\lambda},
      \tilde r_\alpha g \cdot \vec v_{-\lambda} \rangle \\
      &= \langle c \tilde r_\alpha \vec v_{-\lambda},
      \tilde r_\alpha g \cdot \vec v_{-\lambda} \rangle 
      + \langle \vec v_{-\lambda},
      \tilde r_\alpha g \cdot \vec v_{-\lambda} \rangle \\      
      &= c \langle \vec v_{-\lambda},
      g \cdot \vec v_{-\lambda} \rangle
      + \langle \vec v_{-\lambda},
      \tilde r_\alpha g \cdot \vec v_{-\lambda} \qquad %\rangle \\
      &= c m_\lambda( g ) + m_\lambda( \tilde r_\alpha g ).
    \end{align*}
    \end{enumerate}
\end{proof}

Let $I_w^Q$ denote the ideal in $\rationals[c_1,\ldots,c_{\ell(v)}]$
corresponding to $X_{w\circ}^v$, by pulling back the ideal defining
$\overline{B_- w B_+} \subseteq G$ 
along the map $\tilde\beta_Q : \AA^{\ell(v)} \to G$. 
In the $G = GL_n$ case, it is easy to find generators for $I_w^Q$,
as follows.
Fulton's theorem \cite{Fulton92} gives generators (a collection of
minors) for the defining ideals of $(B_-\times B_+)$-orbit closures in $M_n$, 
and therefore also in $GL_n$. Pulling these back along $\tilde\beta_Q$, 
we obtain generators for $I_w^Q$. While Fulton's generators were shown
to be a Gr\"obner basis in \cite{KM}, it is not obvious
that their images in $I_w^Q$ are again a Gr\"obner basis
(a case of which is treated in \cite{WooYong}).
In any case the following theorem does not assume $G=GL_n$.

\begin{Theorem}\label{thm:kl}
  Fix $v\in W$, and a reduced word $Q$ for $v$. 
  Then the function $f$ on $\AA^{\ell(v)}$ defined by
  $$ f(c_1,\ldots,c_{\ell(v)}) :=
  \prod_\omega m_\omega(\tilde\beta_Q(c_1,\ldots,c_{\ell(v)}))\qquad
  \text{$\omega$ ranging over $G$'s fundamental weights} $$
%  (product taken over the fundamental weights $\omega$ of $G$)
  is of degree $\ell(v)$, and its lex-initial term is
  $\prod_i c_i$.

  Under the identification of $\AA^{\ell(v)}$ with $X^v_\circ$,
  the divisor $f=0$ is the preimage of $\Union_\alpha X_{r_\alpha}$.
  By decomposing and intersecting repeatedly, we can produce
  all the other $X_{w\circ}^v$ from this divisor. 
  If $I_w^Q$ is the ideal in $\rationals[c_1,\ldots,c_{\ell(v)}]$
  corresponding to $X_{w\circ}^v$, then $\init I_w^Q$ is Stanley-Reisner.

  We can produce a Gr\"obner basis for $I_w^Q$ by concatenating
  Gr\"obner bases for $I_{w'}^Q$, with $w' \leq w$ in Bruhat order,
  and $w'$ basic in {\em opposite} Bruhat order on $W$.
  (The basic elements of Bruhat orders were computed 
  in \cite{biGrassmannian,GeckKim}.)
\end{Theorem}

\begin{proof}
  More specifically, we claim that for each $\omega$,
  $\init\, m_\omega(\tilde\beta_Q(c_1,\ldots,c_{\ell(v)}))
  = \prod_{i:\, \langle \omega, \alpha_i \rangle = 1} c_i$.
  (Recall that for each $\omega$, there exists a unique $\alpha$ with
  $\langle \omega, \alpha \rangle \neq 0$, and in that case
  $\langle \omega, \alpha \rangle = 1$. So the product of these
  $\prod_{i:\, \langle \omega, \alpha_i \rangle = 1} c_i$
  will be the desired $\prod_i c_i$.)

  This is proven by induction on $\ell(v)$, as follows.
  Let $Q = \alpha_1 Q'$, where $Q'$ is therefore a reduced word for
  $v' := r_{\alpha_1} v < v$.
  If $\langle \omega, \alpha_1 \rangle = 0$, then by 
  lemma \ref{lem:momega} part (2), 
  $$ m_\omega(\tilde\beta_Q(c_1,c_2,\ldots,c_{\ell(v)}))
  = m_\omega(\tilde\beta_{Q'}(c_2,\ldots,c_{\ell(v)})). $$
  If $\langle \omega, \alpha_1 \rangle = 1$, then by 
  lemma \ref{lem:momega} part (3), 
  $$ \init_{c_1} m_\omega(\tilde\beta_Q(c_1,c_2,\ldots,c_{\ell(v)}))
  =  c_1 m_\omega(\tilde\beta_{Q'}(c_2,\ldots,c_{\ell(v)})) $$
  (where $\init_{c_1}$ is defined as in the proof of
  proposition \ref{prop:algdegen}).
  So $\init m_\omega(\tilde\beta_Q(c_1,c_2,\ldots,c_{\ell(v)}))$
  is either $\init m_\omega(\tilde\beta_{Q'}(c_2,\ldots,c_{\ell(v)}))$
  or $c_1$ times that, and chaining these together,
  we get that $\init\, m_\omega(\tilde\beta_Q(c_1,\ldots,c_{\ell(v)})) =
  \prod_{i:\, \langle \omega, \alpha_i \rangle = 1} c_i$.

  The identification $\AA^{\ell(v)} \to X^v_\circ$ is given by
  $(c_1,\ldots,c_{\ell(v)}) 
  \mapsto \tilde\beta_Q(c_1,\ldots,c_{\ell(v)})\cdot B/B$.
  By lemma \ref{lem:momega} part (1), the preimage of $X_{r_\alpha}$
  is given by $m_\omega(\tilde\beta_Q(c_1,\ldots,c_{\ell(v)})) = 0$.
  Hence the preimage of $\Union_\alpha X_{r_\alpha}$ is given by $f=0$.

  Just as in the case of matrix Schubert varieties, 
  we can obtain all $X_w$ by intersecting/decomposing from 
  $\Union_\alpha X_\alpha$, because each $X_w$ is a component of
  $\bigcap_{w' < w} X_{w'}$. Nothing changes when we intersect with $X^v$
  (essentially because $w\leq v$ and $w'<w$ so $w'<v$; the necessary
  $\{X_{w'\circ}^v\}$ are thus once again available by induction).

  Now apply theorem \ref{thm:initprod}, to see that each
  $\init\, I_w^Q$ is Stanley-Reisner over $\rationals$.
  
  For the Gr\"obner basis statement, we use theorem \ref{thm:Grobner},
  noting that opposite Bruhat order is the relevant one for
  containment of Schubert varieties.
  Every basic element of $\{w' : w' \leq v\}$ is basic in the 
  opposite Bruhat order. Some of the basic $w'$ for the opposite
  Bruhat order may not be basic for this subposet, but adding them to
  the Gr\"obner basis does no harm.
\end{proof}

Theorem \ref{thm:kl} shows that $I_w^Q$ has a Gr\"obner basis 
whose leading terms are squarefree, but does not fully determine it
(except for $w = r_{\alpha_i}$), nor does it even determine the
leading terms, which generate the initial ideal.  
We determined this initial ideal in \cite{schubDegen}; 
it is the Stanley-Reisner ideal of
the ``subword complex'' $\Delta(Q,w)$ of \cite{subword}.
The map $\pi_{f,\init}$ defined in theorem \ref{thm:frobdegen} part (3),
from the set of coordinate spaces in $\AA^n$ to the set of
compatibly split subvarieties of $X^v_\circ$, is just the
map taking a subword of $Q$ to its Demazure/nil Hecke product.
Then the order-preserving property of $\pi_{f,\init}$ is a standard
characterization of the opposite Bruhat order in terms of existence of subwords.

As in \cite{Deodhar}, this result and its proofs
are the same if $G$ is taken to be a Kac-Moody Lie group;
even though $G$ is infinite-dimensional, $X^v_\circ$ is still
only $\ell(v)$-dimensional. Unfortunately, it does {\em not} thereby apply 
to the varieties $X_1^\circ \cap X^v$ 
when the big cell $X_1^\circ$ is infinite-dimensional.
This is a shame, as these varieties include 
nilpotent orbit closures \cite{Lusztig}, as nicely recounted in \cite{Magyar}.

We compute a sample $f$, where $Q = r_1 r_2 r_3 r_2$
is a reduced word for $v = 2431$ in the Weyl group $S_4$.
Then $m_{\omega_i} :G\to \AA^1$ is the upper left
$i\times i$ minor, and $\tilde\beta_Q(c_1,c_2,c_3,c_4)$ is
$$
\begin{bmatrix}
  c_1&-1 &   &   \\
   1 & 0 &   &   \\
     &   & 1 &   \\
     &   &   & 1 
\end{bmatrix}
\begin{bmatrix}
  1  &   &   &   \\
     &c_2&-1 &   \\
     & 1 & 0 &   \\
     &   &   & 1 
\end{bmatrix}
\begin{bmatrix}
  1  &   &   &   \\
     & 1 &   &   \\
     &   &c_3&-1 \\
     &   & 1 & 0 
\end{bmatrix}
\begin{bmatrix}
  1  &   &   &   \\
     &c_4&-1 &   \\
     & 1 & 0 &   \\
     &   &   & 1 
\end{bmatrix}
=
%| c_1 -c_2c_4+c_3 c_2 -1 |
%| 1   0           0   0  |
%| 0   c_4         -1  0  |
%| 0   1           0   0  |
\begin{bmatrix}
  c_1&c_3-c_2 c_4&c_2 &-1  \\
  1  & 0 &0 &0  \\
   0& c_4 &-1 &0  \\
   0 & 1 & 0 &0 
\end{bmatrix},
$$
with $m_{\omega_1} = c_1$, $m_{\omega_2} = c_2c_4 - c_3$, and
$m_{\omega_3} = c_3$. Note that they need not be homogeneous;
homogeneity should only be expected when $v$ is $321$-avoiding,
as that is the condition (more generally known as ``$\lambda$-cominuscule'';
see e.g. \cite{Stembridge})
for the $T$-action on $X^v_\circ$ to contain the scaling action.\footnote{%
  {\em Proof.} 
  The $T$-weights on $X^v_\circ$ are $\{x_i-x_j : i<j, \pi(i)>\pi(j)\}$.
  $\pi$ contains the pattern $321$ iff $\exists i < j < k$ 
  with $\pi(i)>\pi(j)>\pi(k)$.
  When that is the case, the weights on $X^v_\circ$ include
  $x_i-x_j,x_j-x_k,x_i-x_k$, and no linear functional on $T^*$ can
  take value $1$ on all three. 
  If $\pi$ is $321$-avoiding, let $\Gm\to T$ take $z$ to the diagonal
  matrix $D$ with $D_{ii} = z$ if $\exists j>i$, $\pi(j)<\pi(i)$ 
  and $D_{ii} = 1$ otherwise; this $\Gm$ acts by dilation on $X^v_\circ$.}

\begin{figure}[h]
  \centering
  \epsfig{file=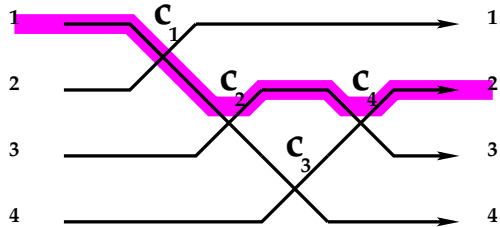,height=1.2in}
  \caption{A graphical way to compute $\tilde \beta_Q(c_1,c_2,c_3,c_4)$, 
    using the wiring diagram of $Q=1232$. Each path from $i$ to $j$
    contributes a term to the $(i,j)$ entry, with a factor of $-1$ for
    each step down and a factor of $c_i$ for each avoidance of a $\times$
    in favor of going through the $c_i$ atop it. The path pictured contributes
    $-1 * c_2 * c_4$ to the $(1,2)$ entry. We invite the reader to redo
    the matrix calculation above using this diagram.}
  \label{fig:wiring}
\end{figure}

Matrix Schubert varieties in $M_n$ are in fact Kazhdan-Lusztig
varieties in $GL_{2n}/B_{2n}$, a fact Fulton used in \cite{Fulton92}
to show that matrix Schubert varieties are normal, Cohen-Macaulay,
and have rational singularities. 
Taking $Q$ to be the ``square word'' 
$$ Q = (r_n r_{n-1} \cdots r_2 r_1)(r_{n+1} r_n \cdots r_3 r_2) 
\cdots (r_{2n-1} \cdots r_n) $$
in $S_{2n}$ from \cite[example 5.1]{subword}, a straightforward computation 
yields
$$ \tilde\beta_Q(c_{n,1},c_{n-1,1},\ldots,c_{1,1},\ \ c_{n,2},\ldots, c_{1,2},
\ \ \ldots,\ \ c_{n,n},\ldots,c_{1,n}) = 
\begin{bmatrix}
  C \cdot D
  & (-1)^n I_n \\
  I_n & 0_n
\end{bmatrix}
$$
where $C$ is the matrix of indeterminates $c_{ij}$, 
the matrices $I_n,0_n$ are the identity and zero matrices of size $n$, 
and $D$ is the diagonal matrix with alternating signs $D_{ii} = (-1)^{i-1}$.
Then the $2n-1$ minors $m_{\omega_i}$ are, up to signs, 
the $(d_{[1,i]})$ and $(d_{[i,n]})$ considered in
section \ref{ssec:matrixSchubs}. These are homogeneous (being determinants), 
and the corresponding v = $(n+1) (n+2) \ldots (2n) 1 2 3 \ldots n$ is
indeed $321$-avoiding.

As stated, theorem \ref{thm:kl} is about Kazhdan-Lusztig varieties
$X_w \cap X^v_\circ$ in a full flag manifold $G/B$. If $P\geq B$ is a 
parabolic subgroup and $v$ is minimal in its $W/W_P$ coset, then the
composite map $X^v_\circ \into G/B \onto G/P$ is an isomorphism of
opposite Schubert cells. If $w$ is also minimal in its coset, 
this restricts to an isomorphism of Kazhdan-Lusztig varieties. 
For example, to study a neighborhood on a Schubert variety 
$X_{w W_P} \subseteq G/P$ centered at the most singular point $w_0 P/P$,
we can apply the theorem with $v = w_0 w_0^P$ and $w$ minimal in its coset.
The matrix Schubert variety case just described is almost an example 
of this, except that $w$ is not minimal in its coset.

\subsection{Multiplicity-free divisors on $G/B$ are splittable}
\label{ssec:multfree}

A reduced subscheme $X \subseteq G/B$ is called \defn{multiplicity-free} in
\cite{BrionMultFree} if the expansion $[X] = \sum_\pi c_\pi [X_\pi]$
of its Chow class in the basis of Schubert classes $[X_\pi],\pi\in W$
has $c_\pi = 0,1$. Brion proves many wonderful geometric facts 
about such subschemes, when they are irreducible. We use theorems
\ref{thm:schemeversion} and \ref{thm:kl} to relate this property
to Frobenius splitting.

\begin{Proposition}\label{prop:multfreesplit}
  Let $X$ be a multiplicity-free divisor on a flag manifold $G/B$.
  Then there is a Frobenius splitting of $G/B$ that compatibly splits $X$.
  If $X$ does not contain the Schubert point $w_0 B/B$, then the 
  splitting can also be made to compatibly split all the Schubert varieties.
\end{Proposition}

\begin{proof}
  We may assume that $X$ does not contain the Schubert point, by using
  some $g\in G$ to move some point outside $X$ to $w_0 B/B$.
  (This may require extending the base field so $G/B$ has closed points 
  outside $X$.)

  Write $\alpha \in [X]$ if the Schubert class $[X_\alpha]$ appears 
  in the expansion of the Chow class $[X]$.
  Let $Y = X \union (\Union_\alpha X_\alpha) 
  \union (\Union_{\alpha\notin [X]} w_0\cdot X_\alpha)$.
  Since $X$ does not contain the Schubert point, it does not contain
  any of the Schubert divisors $X_\alpha$.
  If $X$ contains an opposite Schubert divisor $w_0 \cdot X_\alpha$, then 
  $\alpha \in [X]$. Hence no component listed in this union equals any other.
  The Chow class of this sum is therefore the sum of the
  classes of the terms, hence $2 \sum_\alpha [X_\alpha]$, 
  the anticanonical class.

  Now apply theorem \ref{thm:schemeversion}, part (2), and theorem
  \ref{thm:kl} with $v=w_0$, to see that $Y$ defines a
  Frobenius splitting on $G/B$ with respect to which $Y$ is 
  compatibly split. Since $X$ is a union of components of $Y$,
  it is also split.

  If we actually care about splitting the original $X$, rather than
  $g\cdot X$, we can split using $g^{-1}\cdot Y$ instead. 
  If $X$ doesn't contain the Schubert point, then we can take $g=1$.
\end{proof}

This raises the question of whether this proposition holds for
any multiplicity-free subscheme, not just divisors. 
Note that this proposition does not require $X$ to be irreducible, 
though Brion gives counterexamples showing that some of his results
depend on irreducibilty.

\junk{
\subsection{A coarser decomposition of Bruhat cells}

The proof of theorem \ref{thm:countdegen} actually shows that
there is a discontinuous injection of $X$ into $X'$;
on $\Lambda'\times L$ it is the identity, and on the rest it is
the linear projection to $H$. 

In \cite[]{automatically} we studied geometric vertex decompositions
of \emph{Kazhdan-Lusztig varieties} $\{X_{w\circ}^v\}$.
By iterating these,

Deodhar's point-count in \cite[Corollary 1.2]{Deodhar} shows...

{If Deodhar is picking out a subcomplex of the subword complex,
  that's a sign that $\Lambda > \Lambda'$ sometimes,
  and maybe I don't want to deal with that here.}
}

\bibliographystyle{alpha}    % it seems this does nothing.

\end{document}